\documentclass[10pt,a4paper]{article}
\usepackage{amsmath,amssymb,amsthm,esint,bm}
\usepackage{mathrsfs}
\usepackage{bookmark}
\allowdisplaybreaks[3]

\usepackage[utf8]{inputenc}   
\usepackage[T1]{fontenc}      
\usepackage{lmodern}

\newcommand{\ainc}[1]{\hyperref[ainc]{{\normalfont(aInc){\ensuremath{_{#1}}}}}}
\newcommand{\adec}[1]{\hyperref[adec]{{\normalfont(aDec){\ensuremath{_{#1}}}}}}
\newcommand{\inc}[1]{\hyperref[inc]{{\normalfont(Inc){\ensuremath{_{#1}}}}}}
\newcommand{\dec}[1]{\hyperref[dec]{{\normalfont(Dec){\ensuremath{_{#1}}}}}}

\allowdisplaybreaks

\usepackage[T1]{fontenc}
\usepackage[utf8]{inputenc}
\usepackage{authblk}
\usepackage{amsmath,amssymb,amsthm,esint,bm}
\usepackage{mathrsfs}
\usepackage{bookmark}
\usepackage{amsmath}

\newtheorem{theorem}{Theorem}[section]
\newtheorem{lemma}[theorem]{Lemma}

\theoremstyle{definition}

\newtheorem{definition}[theorem]{Definition}

\numberwithin{equation}{section}

\newcommand{\ls}{\leqslant}
\newcommand{\rs}{\geqslant}

\title{Riesz potential estimates for non-linear elliptic  obstacle problems}

\author[a]{Qi Xiong}
\author[b]{Zhenqiu Zhang\thanks{Corresponding author.}}
\author[c]{Lingwei Ma}

\affil[a]{School of Mathematics, Southwest Jiaotong University, Chengdu, Sichuan, 610031, P.R. China}
\affil[b]{School of Mathematical Sciences and LPMC, Nankai University, Tianjin, 300071, P.R. China}
\affil[c]{School of Mathematical Sciences and LPMC, Nankai University, Tianjin, 300071, P.R. China}

\date{\today}

\usepackage{hyperref}
\begin{document}
\maketitle
\footnotetext[1]{E-mail: xq@swjtu.edu.cn(Q. Xiong),  zqzhang@nankai.edu.cn (Z. Zhang),malingwei@nankai.edu.cn (L. Ma).}

\maketitle
\begin{abstract}
This paper investigates a class of nonlinear elliptic obstacle problems involving measure data, where the nonlinearity depends on the variable of the possibly unbounded solution itself. We establish pointwise gradient estimates for the solutions in terms of Riesz potentials.


\end{abstract}


\section{Introduction and main results}\label{section1}
\ \ \  In this paper, we consider  the non-homogeneous    obstacle problems  and  they are  related to elliptic equations of the type
\begin{equation} \label{1.1}
  -\operatorname{div}\left({a}(x,u,Du)\right) =\mu \quad\quad\mbox{in}\ \ \ \Omega \\[0.05cm].
\end{equation} 
Here $ \Omega\subseteq \mathbb{R}^n $ with $ n\geqslant2 $ is a bounded open set, and 
$\mu$ belongs to the set $\mathcal{M}_{b}(\Omega)$ of signed Radon measures for which $|\mu|(\Omega)$ is finite, and   $|\mu|$  represents the total variation of $\mu$.
Moreover we
assume that $\mu(\mathbb{R}^n \backslash \Omega)=0$ and $a=a(x,z,\eta): \Omega \times \mathbb{R} \times  \mathbb{R}^n \rightarrow \mathbb{R}^n$  satisfies  the standard ellipticity   and growth conditions:  
\begin{eqnarray}\label{0a}
  \left\{\begin{array}{r@{}c@{}ll}
&&D_{\eta} a(x,z,\eta )\lambda \cdot \lambda \geqslant l(s^2+|\eta|^2)^{\frac{p-2}{2}}|\lambda|^2 \,, \\[0.05cm]
&&|a(x,z,\eta)|+|D_{\eta} a(x,z,\eta )|(s^2+|\eta|^2)^{\frac{1}{2}}\leqslant L(s^2+|\eta|^2)^{\frac{p-1}{2}}\, \\[0.05cm]
  \end{array}\right.
\end{eqnarray}
for a growth exponent $p\geqslant2$ and given constants $0<l\leqslant 1 \leqslant L<+\infty$ and $s \in[0,1]$, where $x \in \Omega, z \in \mathbb{R}, \eta,\lambda \in \mathbb{R}^n$. And $D_{\eta}$ denotes the differentiation in $\eta$.    We assume that $a(x,z,\eta)$ is Dini-continuous in $x$ variable 
in the sense that
\begin{equation}\label{omegaa}
	|a(x,z,\eta) - a(x_0,z,\eta)| \leq 2L\omega(|x-x_0|)(s^2+|\eta|^2)^{\frac{p-1}{2}},
\end{equation}
where $\omega:[0,\infty) \to [0,1]$ is a nondecreasing 
modulus of continuity that satisfies 
\begin{equation*} 
	\int_0^R \frac{\omega(\rho)d\rho}{\rho} < \infty
\end{equation*}
for every $R > 0$.

Moreover, we assume that $a(x,z,\eta)$ has a natural Hölder continuity assumption 
on the solution variable $z$:

\begin{equation}\label{omegaz}
	|a(x,z_1,\eta) - a(x,z_2,\eta)| \leq \Gamma^{p-1}|z_1-z_2|^{(p-1)\sigma}(s^2+|\eta|^2)^{\frac{p-1}{2}}
\end{equation}
for some $\Gamma \in [0,\infty)$. The range of the 
constant $\sigma \in \left(0, \frac{1}{p}-\frac{1}{\sigma_1}\right] $, $\sigma_1 = \sigma_1(n,p,l,L) \in (p,p+1)$ is  provided by Lemma 1.1 in \cite{kim}.

 The obstacle condition that we impose on the solutions is  of the form $u \geq \psi$ a.e.  in $\Omega$, where $\psi \in W^{2,1}(\Omega)\cap W^{1,p}(\Omega)$ 
 is  given function which satisfies $|D\psi|^{p-2}|D^2\psi| \in L^1_{loc}(\Omega) $. And we assume that $D_x a(x,u,D\psi)\in L^1_{loc}(\Omega), D_u a(x,u,D\psi)\cdot Du\in L^1_{loc}(\Omega)$, so that 
  $\operatorname{div}\left({a}(x,u,D\psi)\right) \in L^{1}_{loc}(\Omega)$.  The obstacle problem can be formulated by the variational inequality
  \begin{equation}\label{fjd}
  	\int_{\Omega} a(x,u,Du)\cdot D(v-u)\,dx \geq \int_{\Omega} (v-u)\,d\mu
  \end{equation}
  for every functions \(v \in u+W_0^{1,p}(\Omega)\) with \(v \geq \psi\) a.e. in \(\Omega\). For a precise definition, please refer to Definition~\ref{opdy}.

The obstacle problem is a fundamental model in the study of variational inequalities and free boundary problems. Originally introduced in the context of classical linear elasticity to describe the equilibrium position of an elastic membrane constrained by an obstacle, it has since become central in both analysis and applications. Beyond its geometric connections to minimal surfaces and capacity theory, the obstacle problem naturally extends to various functional settings, including Hölder, Lebesgue, and Orlicz spaces. Its applications range from fluid filtration in porous media and elasto-plasticity to optimal control, phase transitions, and financial mathematics; see \cite{c1,f1,ks1,r1} for further details.

This work studies how the regularity of solutions to elliptic obstacle problems of $p$-Laplacian type depends on the obstacle and the nonhomogeneous term. In particular, we aim to establish pointwise estimates for the gradients of solutions  by means of   Riesz potentials.

The systematic study of potential estimates originated from the fundamental results of Kilpel$\ddot{a}$inen and Mal$\acute{y}$ \cite{km6,km5}, who established pointwise estimates for solutions to nonlinear elliptic equations of the form \eqref{1.1} in terms of the nonlinear Wolff potential, defined by
\begin{equation*}
	W^{\mu}_{\beta,p}(x,R):=\int_0^R\left( \frac{|\mu|(B_{\rho})}{\rho^{n-\beta p}}\right) ^{1/(p-1)}\frac{\operatorname{d}\!\rho}{\rho}
\end{equation*}
for $\beta \in (0,n]$ and $p>1$. It is worth noting that this potential reduces to the Riesz potential when $p=2$:
\begin{equation*}
	\mathbf{I}^{|\mu|}_{\beta}(x,R):=\int_{0}^{R} \frac{|\mu|(B_{\rho})}{\rho^{n-\beta }}\frac{\operatorname{d}\!\rho}{\rho}.
\end{equation*}
Subsequently, these results were extended to more general settings by Trudinger and Wang \cite{tw7,tw8} employing alternative approaches.
Furthermore,  Mingione \cite{m9} first obtained Riesz potential estimates for gradient of solutions for the case $p=2$:
$$|Du(x)|\leqslant c\textbf{I}^{|\mu|}_{1}(x,R)+c\fint_{B_R(x)}(|Du|+s)dy,$$

Later, Duzaar and Mingione established pointwise gradient estimates for the case $p \geqslant 2$ using Wolff potentials in \cite{dm10}, and   derived Riesz potential estimates for the range $2-\tfrac{1}{n}<p< 2$ in \cite{dm11}. 
Dong and Zhu \cite{dz2} further extended these results to the open case $p \in (1,\tfrac{3n-2}{2n-1}]$.  Moreover, Kim \cite{kim} established Riesz potential estimates for cases where the nonlinearity   depends on the solution.
  Potential estimates are comprehensive tools to obtain regularity estimates. Indeed, they have been intensively studied and extended in several directions; we refer to   \cite{dm21,km13} for the elliptic systems,  \cite{dz1,km15,km16} for the parabolic  equations,                      \cite{b13,cyz,xiong6,xiong1,xiong2,xiong3,xiong4,xiong5} for the  problems with Orlicz growth,   respectively.    For additional findings on potential estimates, please refer to \cite{dm11,km13,km14,m9,xiao}. 

Potential estimates for nonlinear elliptic obstacle problems with measure data were first investigated by Scheven in \cite{s26, s28}, where gradient estimates were established in terms of Wolff potentials. Subsequently, Byun, Song, and Youn \cite{bsy, bsy2} introduced a new form of gradient estimates for constant-coefficient obstacle problems by using Riesz potentials. Building on these advancements, this paper aims to extend such results to the setting where the nonlinearity explicitly depends on the solution itself. Specifically, we derive gradient potential estimates for these more general problems via Riesz potentials.

We now proceed to outline our main results, starting with the presentation of key definition  and assumptions.

Now we describe our obstacle problem  $OP(\psi ; \mu)$.
Let
$$T_k(s):=\max\left\lbrace -k,\min \left\lbrace k,s\right\rbrace\right\rbrace, \ \ \ \forall \ k>0,\quad\forall\ s\in\mathbb{R}.$$

Moreover, for given Dirichlet boundary data $h\in W^{1,p}(\Omega)$, we define
$$\mathcal{T}^{1,p}_{h}(\Omega):=\left\lbrace u: \Omega\rightarrow \mathbb{R} \ \text{measurable}: T_{k}(u-h)\in W_{0}^{1,p}(\Omega) \ \ \text{for \ every} \ k\in(0,\infty) \right\rbrace. $$

\begin{definition}\label{opdy}
	Suppose that an obstacle function $\psi \in W^{1,p}(\Omega)$, measure data $\mu \in \mathcal{M}_{b}(\Omega)$ and boundary data $h \in W^{1,p}(\Omega)$ with $h\rs \psi$ a.e. in $\Omega$ are given. We say that $u \in \mathcal{T}^{1,p}_{h}(\Omega)$ with $u \rs    \psi$ a.e. in $\Omega$ is  a limit of approximating solutions of the obstacle problem $OP(\psi ; \mu)$ if 
	\begin{itemize}
	\item[(i)] there exist functions
	$f_{i} \in W^{-1,p'}(\Omega)\cap L^1(\Omega)$
	satisfying
	\begin{eqnarray*}
		\left\{\begin{array}{r@{}c@{}ll}
			&& f_{i}\stackrel{\ast}\rightharpoonup \mu \ \text{in} \ \mathcal{M}_{b}(\Omega), \ \ \text{as} \ i\rightarrow+\infty \,, \\[0.05cm]
			&&\limsup_{i\rightarrow+\infty}\int_{B_R(x_0)}|f_i|dx\leqslant|\mu|( B_R(x_0)), \   \ \ \forall \  B_R(x_0)\subseteq \Omega \,; \\[0.05cm]
		\end{array}\right.
	\end{eqnarray*}
\item[(ii)] there exist solutions $u_{i}\in W^{1,p}(\Omega)$ with $u_{i}\geqslant \psi$ a.e. in $\Omega$ of the variational inequalities
	\begin{equation*} 
		\int_{\Omega}a(x,Du_i)\cdot D(v-u_{i})dx\geqslant \int_{\Omega}f_{i}(v-u_{i})dx
	\end{equation*}
	for every $ \ v \in u_{i}+W_{0}^{1,p}(\Omega)$ with $v\geqslant \psi$ a.e. in $\Omega$;
	\item[(iii)]  the following convergences hold:
	\begin{equation*}
		\begin{cases}
			
			u_{i} \rightarrow u \ \ a.e. \ \ \  \text{in} \ \  \Omega,\ \ \text{as} \ i\rightarrow+\infty, \\
			
			u_i\rightarrow u \ \ \  \text{in} \ L^q(\Omega), \ \ \forall \  q \in \left(0,\frac{(p-1)n}{n-p} \right),  \ \ \text{as} \ i\rightarrow+\infty, \\
			
			Du_i\rightarrow Du \ \ \ \text{in} \ L^r(\Omega), \ \ \forall \  r \in \left(0,\frac{(p-1)n}{n-1} \right),\ \ \text{as} \ i\rightarrow+\infty. \\
		\end{cases}
	\end{equation*}
\end{itemize}
\end{definition}
The existence of solutions converging in the sense of Definition 1.1 was established in \cite{s26}.

To facilitate our analysis, throughout this paper we define
\begin{align*}
	 \textbf{I}^{[\psi]}_{\beta }(x,R):= \int_{0}^{R}  \frac{D\Psi(B_{\rho}(x))}{\rho^{n-\beta }}\frac{\operatorname{d}\!\rho}{\rho},
\end{align*}
where the term $D\Psi(B_{\rho}(x))$ is given by $$D\Psi(B_{\rho}(x)):=\int_{B_{\rho}(x)}|\operatorname{div} a(\xi,u,D\psi)|d\xi.$$

Finally, we formulate the main results of this paper. We first establish a gradient potential estimate in terms of Riesz potentials.
\begin{theorem}\label{thm}
	Under the assumptions \eqref{0a}, \eqref{omegaa} and \eqref{omegaz}, let $u \in W^{1,1}(\Omega)$ with $u \rs \psi$ a.e. be a limit of approximating solutions to OP($\psi; \mu$) with measure data $\mu \in \mathcal{M}_b(\Omega)$,  assume that $x_0$ is a Lebesgue point of $Du$ and $B_R(x_0) \subseteq \Omega$, and there exists a  constant $c_1 = c_1(n, p, l, L) \ge 1$ such that
	\begin{align}\label{c11}
		c_1 \int_0^R \omega(\rho) \frac{d\rho}{\rho} \le 1.  
	\end{align}
	Then 
\begin{align*}
	|Du(x_0)| 
	&\le c \fint_{B_R(x_0)} (|Du|+s) \, dx + c \left(  \textbf{I}^{|\mu|}_{1 }(x_0,R) + \textbf{I}^{[\psi]}_{1 }(x_0,R) \right)^{\frac{1}{p-1}} \nonumber \\
	&\quad + \left( \frac{c \Gamma R^\sigma}{1 - (\frac{1}{2})^\sigma} \right)^{\frac{n}{\beta \sigma}}\fint_{B_R} (|Du|+s) \, dx \\ \nonumber
	&\quad \cdot \left[ \left( \fint_{B_R(x_0)} (|Du|+s) \, dx \right)^{1 + \frac{n}{\beta}} + \left(  \textbf{I}^{|\mu|}_{1 }(x_0,R) + \textbf{I}^{[\psi]}_{1 }(x_0,R) \right)^{\frac{n}{(p-1)\beta}} \right] ,
\end{align*}
		where  $\beta = \beta(n, p, l, L) \in (0, 1)$   is as in Lemma \ref{w3}  and $c  = c (n, p, l, L)  $.
\end{theorem}
 
By applying Theorem \ref{thm}, we can obtain the VMO regularity and gradient continuity for solutions to the obstacle problem.
\begin{theorem}\label{thm2}
Let the hypotheses of Theorem \ref{thm} be satisfied. Furthermore, suppose that
	\begin{align}
		\sup_{x \in B_{2R}(x_0)} \left[ \int_0^R \frac{|\mu|(B_\rho(x))}{\rho^{n-1}} \frac{d\rho}{\rho} + \int_0^R \frac{D\Phi(B_\rho(x))}{\rho^{n-1}} \frac{d\rho}{\rho} \right] < \infty \label{su1}
	\end{align}
	and
	\begin{align}
		\lim_{r \to 0} \sup_{x \in B_R(x_0)} \left[ \frac{|\mu|(B_r(x))}{r^{n-1}} + \frac{D\Phi(B_r(x))}{r^{n-1}} \right] = 0. \label{su2}
	\end{align}
	Then $Du$ is $VMO$-regular in $B_R(x_0)$ whenever $B_{3R}(x_0) \subseteq \Omega$.
\end{theorem}

\begin{theorem}\label{thm3}
	Under the hypotheses of Theorem \ref{thm}, assume that
	\begin{align*}
		\lim_{R \to 0} \sup_{x \in B_{2R}(x_0)} \left( \int_0^R \frac{|\mu|(B_\rho(x))}{\rho^{n-1}} \frac{d\rho}{\rho} + \int_0^R \frac{D\Phi(B_\rho(x))}{\rho^{n-1}} \frac{d\rho}{\rho} \right) = 0.
	\end{align*}
	Then $Du$ is continuous in $B_R(x_0)$ whenever $B_{3R}(x_0) \subseteq\Omega$.
\end{theorem}

The remainder of this paper is organized as follows. Section 2 contains some notions and   some comparison estimates.  In Section 3, we complete the proof of several theorems.

\section{Preliminaries and comparison estimates} 
\hskip\parindent In this section, we present several preliminary results and establish some key comparison estimates. Utilizing these estimates, we derive an analogous excess decay estimate for solutions to obstacle problems involving measure data.

 For an integrable map $f: \Omega \rightarrow \mathbb{R}^m $, we write
$$(f)_{\Omega}:=\fint_{\Omega}fdx:=\frac{1}{|\Omega|}\int_{\Omega}fdx.$$
For $q\in[1,\infty)$, it is easily verified that
\begin{equation} \label{1.8}
	\parallel f-(f)_{\Omega}\parallel_{L^q(\Omega)}\leqslant2\min_{c\in \mathbb{R}^m}\parallel f-c\parallel_{L^q(\Omega)}.
\end{equation}

Next we want to obtain some comparison estimates between   the solutions to   obstacle problems and  to homogeneous elliptic equations.   

To facilitate our analysis, throughout this paper we let $B_{4R}(x_0)\subseteq \Omega$  and     $B_{4R}:=B_{4R}(x_0)$. For simplicity, the center $x_0$ of the balls will be omitted in what follows. Assume that $u \in W^{1,1}(\Omega)$ with $u \rs \psi $ a.e. is a limit of approximating solutions
to $OP(\psi;\mu)$ with measure data $\mu \in \mathcal{M}_b(\Omega)$, and $w_1, w_2, w_3, w_4, w_5$
satisfy separately
 \begin{itemize}
 	\item  {for $ {w}_1$:}
\begin{equation*}
	\left\{\begin{array}{r@{\ \ }c@{\ \ }ll}
		&\int_{B_{4R}}  a(x, u, Dw_1) \cdot D(v-w_1)dx \geqslant 0,\ \ &\mbox{for every} \ v \in w_1+W_{0}^{1,p}(B_{4R}) \ \  \mbox{with} \ v\geqslant \psi \  a.e. \ \mbox{in}\ \ B_{4R} \,, \\[0.05cm]
		&w_1=u \ \ \ \ \ \ \ \ \ \  \ \ &\mbox{on}\ \ \partial B_{4R} \,;
	\end{array}\right.
\end{equation*}
\item  {for $ {w}_2$:}
\begin{equation*}
	\left\{\begin{array}{r@{\ \ }c@{\ \ }ll}
		-\operatorname{div} a(x, u, Dw_2) &=& -\operatorname{div} a(x, u, D\psi)\ \ \ \  \ &\mbox{in}\ \ B_{4R}\,, \\[0.05cm]
		w_2&=&w_1 \ \ \ \ \ \ \ \ \ \ \ \ \ \ \  \ \ \ \ \  \ \ \ \ \  &\mbox{on}\ \   \partial B_{4R} \,;
	\end{array}\right.
\end{equation*}
\item  {for $ {w}_3$:}
\begin{equation*}
	\left\{\begin{array}{r@{\ \ }c@{\ \ }ll}
		-\operatorname{div} a(x, u, Dw_3) &=& 0\ \ \ \ \ \ \ \ \ \  \ \ \  \ \ \ \ \ \    \ &\mbox{in}\ \ B_{4R}\,, \\[0.05cm]
		w_3&=&w_2 \ \ \ \ \ \ \ \ \ \  \ \ &\mbox{on}\ \ \partial B_{4R} \,;
	\end{array}\right.
\end{equation*}
\item  {for $ {w}_4$:}
\begin{equation*}
	\left\{\begin{array}{r@{\ \ }c@{\ \ }ll}
		-\operatorname{div} a(x, (u)_{B_{2R}}, Dw_4)&=& 0\ \ \ \ \ \  \ \ \ \ \  \ \  \  \ &\mbox{in}\ \ B_{2R}\,, \\[0.05cm]
		w_4&=&w_3 \ \ \ \ \ \ \ \ \ \  \ \ &\mbox{on}\ \ \partial B_{2R} \,;
	\end{array}\right.
\end{equation*}
\item  {for $ {w}_5$:}
\begin{equation*}
	\left\{\begin{array}{r@{\ \ }c@{\ \ }ll}
		-\operatorname{div} a(x_0, (u)_{B_{2R}}, Dw_5)&=& 0\ \ \ \ \ \  \ \ \ \ \ \ \ \   \ &\mbox{in}\ \ B_{R}\,, \\[0.05cm]
		w_5&=&w_4 \ \ \ \ \ \ \ \ \ \  \ \ &\mbox{on}\ \ \partial B_{R} \,.
	\end{array}\right.
\end{equation*}
\end{itemize}
Let \begin{equation}\label{pm}
	 0 < q < p_m := \min\left\{ p, \frac{n(p-1)}{n-1} \right\}.
\end{equation}
 According to the same proof of Lemma 3.5 in \cite{s26}, we establish the following comparison estimate:
\begin{equation}\label{duw1}
	\fint_{B_{4R}} |Du-Dw_1|^q \, dx \leq c_1 \left[ \frac{|\mu|(B_{4R})}{(4R)^{n-1}} \right]^{\frac{q}{p-1}},
\end{equation}
where $c_1=c_1(n,p,q,l,L)$.

 Following the same proofs of Lemma 3.7 and Lemma 3.8 in\cite{s26}, we obtain the following inequalities
\begin{equation}\label{w1w2}
	\fint_{B_{4R}} |Dw_1 - Dw_2|^q dx \leq c_2 \left[ \frac{D\Psi(B_{4R})}{(4R)^{n-1}} \right]^{\frac{q}{p-1}}
\end{equation}
and
\begin{equation}\label{w2w3}
	\fint_{B_{4R}} |Dw_2 - Dw_3|^q dx \leq c_2 \left[ \frac{D\Psi(B_{4R})}{(4R)^{n-1}} \right]^{\frac{q}{p-1}},
\end{equation}
where $c_2=c_2(n,p,q,l,L)$.

Applying Lemma 3.3 and Lemma 3.4 in \cite{kim}, we obtain the following collection of comparison estimates
\begin{equation}\label{w3w4}
	\fint_{B_{2R}} |Dw_3-Dw_4|^p dx \leq \left(c_3 \Gamma R^{\sigma} \right)^p \left(\fint_{B_{4R}}(|Du| + |Dw_3| + s) \, dx \right)^{p(1+\sigma)},
\end{equation}
\begin{equation}\label{w4w5}
	\fint_{B_{R}} |Dw_4 - Dw_5|^p dx \leq c_4 [\omega(R)]^2 \left( \fint_{B_{2R}}(|Dw_4| + s) \, dx \right)^p
\end{equation}
and
\begin{equation}\label{w4w5---}
	\fint_{B_{R}} \left( |Dw_4|^2 + |Dw_5|^2+s^2 \right)^{\frac{p-2}{2}} |Dw_4-Dw_5|^2 dx \leq c_4 [\omega(R)]^2 \fint_{B_{R}} (|Dw_4| + s)^p dx,
\end{equation}
where the constants $c_3,c_4$ both depend on $n,p,l,L$.

The following lemma is a direct consequence of Theorem 18 in \cite{km155}. 
\begin{lemma}\label{w3}
 There exist constants   $\beta \in (0,1)$, $c\rs1$, both depending on $n, p, l, L$, such that   
 
	\begin{equation*}
		\fint_{B_{4\delta R}}|w_3 - (w_3)_{B_{4\delta R}}| \, dx \leq \delta^\beta \fint_{B_{4R}} |w_3 - (w_3)_{B_{4R}}| \, dx + c \delta^\beta R s,
	\end{equation*}
\mbox{for every} $   \delta \in (0, 1)$.
\end{lemma}
Following the arguments in \cite[Lemma 2.1]{kim}, we obtain the following result.
\begin{lemma}\label{w5}
There exist a universal constant $\delta_1 \in (0,  1)$, depending on $n,p,l,L$, such that 
	\begin{equation*}
		\fint_{B_{\frac{\delta R}{2}}} |Dw_5 - (Dw_5)_{B_{2\delta R}}| \, dx \leq \frac{1}{2^{2n+10}} \fint_{B_R} |Dw_5 - (Dw_5)_{B_R}| \, dx
	\end{equation*} 
\mbox{for every} $   \delta \in (0, \delta_1]$.
\end{lemma}
The subsequent analysis relies on the following technical lemma, which can be found in \cite[Lemma 2.2]{kim}.
\begin{lemma}\label{w4}
Assume that there exist a universal constant $\delta_2 \in (0, \delta_1]$ and $c_5 \ge 1$ such that 
	\begin{equation*}
		(|Dw_4| + s)_{B_{2R}} \leq 2^{n+10} \quad \text{and} \quad \int_0^{2R}\frac{\omega(\rho)}{\rho} \, d\rho \leq \frac{1}{c_5}.
	\end{equation*}
	Then
	\begin{equation*}
		\sup_{x,y \in B_{2\delta_2 R}(x_0)} |Dw_4(x) - Dw_4(y)| \leq \frac{1}{2^{10}},
	\end{equation*}
	where $\delta_2=\delta_2(n,p,l,L)$.
\end{lemma}

Next, we choose 
\begin{equation} \label{delta}
	\delta = \min \{\delta_1, \delta_2\} \in (0, 1),
\end{equation}
where the constants $\delta_1, \delta_2,$   are given by  Lemma \ref{w5}, and Lemma \ref{w4}, respectively. 
 
To properly treat the constant $\Gamma$ in \eqref{omegaz}, we establish the following lemma, which will be essential for our subsequent estimates.

\begin{lemma}\label{uuosc}
For the constant $\delta \in (0, 1)$ in \eqref{delta}, we derive
\begin{align*}
	\fint_{B_{4\delta R}} |u - (u)_{B_{4\delta R}}| \, dx &\leq 4\delta^\beta \fint_{B_{4R}} |u - (u)_{B_{4R}}| \, dx + c \delta^{-n} R \left[ \frac{|\mu|(B_{4R})}{(4R)^{n-1}} \right]^{\frac{1}{p-1}} \\
	&\quad + c\delta^{-n} R\left[ \frac{D\Psi(B_{4R})}{(4R)^{n-1}} \right]^{\frac{1}{p-1}} + c \delta^\beta R s
\end{align*}
and
\begin{align*}
	R \fint_{B_R} |Du| \, dx\leq c \fint_{B_{4R}} |u - (u)_{B_{4R}}| \, dx + c R \left[ \frac{|\mu|(B_{4R})}{(4R)^{n-1}} \right]^{\frac{1}{p-1}} + c R \left[ \frac{D\Psi(B_{4R})}{(4R)^{n-1}} \right]^{\frac{1}{p-1}},
\end{align*}
where $\beta \in (0,1)$ is as in Lemma \ref{w3} and $c=c(n,p,l,L)$.
\end{lemma}

\begin{proof}
According to \eqref{1.8} and Lemma \ref{w3}, we conclude that
\begin{align*}
	\fint_{B_{4\delta R}} |u - (u)_{B_{4\delta R}}| \, dx &\leq 2\fint_{B_{4\delta R}} |u - (w_3)_{B_{4\delta R}}| \, dx \\
	&\leq 2\fint_{B_{4\delta R}} |u - w_3| \, dx + 2\fint_{B_{4\delta R}} |w_3 - (w_3)_{B_{4\delta R}}| \, dx \\
	&\leq 2\delta^\beta \fint_{B_{4R}} |w_3 - (w_3)_{B_{4R}}| \, dx + c \delta^{-n} \fint_{B_{4R}} |u - w_3| \, dx + c \delta^\beta R s \\
	&\leq 4\delta^\beta \fint_{B_{4R}} |u - (u)_{B_{4R}}| \, dx + c \delta^{-n} \fint_{B_{4R}} |u - w_3| \, dx + c \delta^\beta R s.
\end{align*}
Next, we make use of Poincar\'{e}'s inequality and  \eqref{duw1}, \eqref{w1w2}, \eqref{w2w3} to get
\begin{align*}
	\fint_{B_{4R}} |u - w_3| \, dx &\leq c \fint_{B_{4R}} |u - w_1| + |w_1 - w_2| + |w_2 - w_3| \, dx \\
	&\leq cR \fint_{B_{4R}} |Du - Dw_1| + |Dw_1 - Dw_2| + |Dw_2-Dw_3| \, dx \\
	&\leq cR\left[ \frac{|\mu|(B_{4R})}{(4R)^{n-1}} \right]^{\frac{1}{p-1}}+R\left[ \frac{D\Psi(B_{4R})}{(4R)^{n-1}} \right]^{\frac{1}{p-1}}.	
\end{align*}
The combination of the above two inequalities yields
\begin{align*}
	\fint_{B_{4\delta R}} |u - (u)_{B_{4\delta R}}| \, dx &\leq 4\delta^\beta \fint_{B_{4R}} |u - (u)_{B_{4R}}| \, dx + c  \delta^{-n} R \left[ \frac{|\mu|(B_{4R})}{(4R)^{n-1}} \right]^{\frac{1}{p-1}} \\
	&\quad + c  \delta^{-n} R  \left[ \frac{D\Psi(B_{4R})}{(4R)^{n-1}} \right]^{\frac{1}{p-1}} + c \delta^\beta R s.
\end{align*}
Following the same arguments as in Lemma 3.5 of \cite{kim}, we obtain
\begin{align*}
	R \fint_{B_R} |Du| \, dx &\leq R \fint_{B_{R}} |Du - Dw_3|dx+R\fint_{B_{R}} |Dw_3|dx \\
	&\leq c R \fint_{B_{4R}} |Du - Dw_3|dx+c\fint_{B_{4R}} |w_3-(w_3)_{B_{4R}}|dx \\
	&\leq c R \fint_{B_{4R}} |Du - Dw_3| \, dx + c \fint_{B_{4R}} |u - w_3| \, dx + c \fint_{B_{4R}} |u - (u)_{B_{4R}}| \, dx \\
	&\leq c \fint_{B_{4R}} |u - (u)_{B_{4R}}| \, dx + c R \left[ \frac{|\mu|(B_{4R})}{(4R)^{n-1}} \right]^{\frac{1}{p-1}} + c R \left[ \frac{D\Psi(B_{4R})}{(4R)^{n-1}} \right]^{\frac{1}{p-1}}.
\end{align*}
This completes the proof.
\end{proof}
In what follows, we aim at deriving an excess decay estimate  in terms of Wolff potentials.
\begin{lemma}\label{du-6}
	Suppose that 
	$$(|Du|+s)_{B_{4R}}\ls 3, \ \  \frac{|\mu|(B_{4R})}{(4R)^{n-1}}\ls 1, \ \  \frac{D\Psi(B_{4R})}{(4R)^{n-1}}\ls 1, \ \  \omega(4R)\ls 1, \ \ \Gamma(4R)^{\sigma}\ls 1.$$
Then we have
\begin{align*}
	\fint_{B_{2\delta R}}|Du-(Du)_{B_{2\delta R}}|dx 
	&\ls \frac{1}{2^7} \fint_{B_{4 R}}|Du-(Du)_{B_{4 R}}|dx\\
	&\quad+c_6 \left\lbrace  \left[ \frac{|\mu|(B_{4R})}{(4R)^{n-1}} \right]^{\frac{1}{p-1}} +    \left[ \frac{D\Psi(B_{4R})}{(4R)^{n-1}} \right]^{\frac{1}{p-1}}+\omega(4R)^{\frac{2}{p}}+\Gamma(4R)^{\sigma} \right\rbrace.
\end{align*}	
\end{lemma}

\begin{proof}
By virtue of the estimates \eqref{duw1}, \eqref{w1w2}, \eqref{w2w3}, \eqref{w3w4}  and \eqref{w4w5}, we arrive at
	\begin{align*}
		\fint_{B_R}|Du-Dw_5|dx &\ls \fint_{B_R}|Du-Dw_1|dx+\fint_{B_R}|Dw_1-Dw_2|dx+\fint_{B_R}|Dw_2-Dw_3|dx\\ &\quad+\fint_{B_R}|Dw_3-Dw_4|dx+\fint_{B_R}|Dw_4-Dw_5|dx 	\\ 
		&\ls c  \left[ \frac{|\mu|(B_{4R})}{(4R)^{n-1}} \right]^{\frac{1}{p-1}} + c   \left[ \frac{D\Psi(B_{4R})}{(4R)^{n-1}} \right]^{\frac{1}{p-1}}+c\omega(4R)^{\frac{2}{p}}+c\Gamma(4R)^{\sigma}.
	\end{align*}
Combining Lemma \ref{w5} with the above inequality, we obtain
\begin{align*}
	\fint_{B_{2\delta R}}|Du-(Du)_{B_{2\delta R}}|dx 	&\ls 2	\fint_{B_{2\delta R}}|Dw_5-(Dw_5)_{B_{2\delta R}}|dx+2\fint_{B_{2\delta R}}|Du-Dw_5|dx \\
		&\ls \frac{1}{2^{2n+9}} \fint_{B_{ R}}|Dw_5-(Dw_5)_{B_{ R}}|dx+c\fint_{B_{ R}}|Du-Dw_5|dx \\
	&\ls \frac{1}{2^{2n+8}} 	\fint_{B_{   R}}|Du-(Du)_{B_{ R}}|dx+c\fint_{B_{R}}|Du-Dw_5|dx \\
	&\ls \frac{1}{2^7} \fint_{B_{4 R}}|Du-(Du)_{B_{4 R}}|dx\\
	&\quad+c_6 \left\lbrace  \left[ \frac{|\mu|(B_{4R})}{(4R)^{n-1}} \right]^{\frac{1}{p-1}} +    \left[ \frac{D\Psi(B_{4R})}{(4R)^{n-1}} \right]^{\frac{1}{p-1}}+\omega(4R)^{\frac{2}{p}}+\Gamma(4R)^{\sigma} \right\rbrace, 
\end{align*}	
which completes the proof.	
\end{proof}
  Given the universal constants $c_1, c_2,$  $c_3$ and $c_5$ appearing in \eqref{duw1}, \eqref{w1w2}, \eqref{w3w4}, and Lemma \ref{w4}, we suppose that
\begin{equation}\label{assume1}
	(|Du|+s)_{B_{4R}} \leq 4, \quad \frac{1}{4} \leq (|Du|+s)_{B_{2\delta R}} \leq 4,  
\end{equation}
\begin{equation}\label{assume2}
	c_1 \left[ \frac{|\mu|(B_{4R})}{(4R)^{n-1}} \right]^{\frac{1}{p-1}} + c_2 \left[ \frac{D\Psi(B_{4R})}{(4R)^{n-1}} \right]^{\frac{1}{p-1}} + c_3 \Gamma(4R)^{\sigma} + \omega(4R)
	+ c_5 \int_0^{4R} \frac{\omega(\rho)}{\rho} d\rho \leq \frac{\delta^{2n}}{2^{2n+16}}.
\end{equation}

\begin{lemma}\label{oscw4}
	Under the assumption  \eqref{assume1} and \eqref{assume2}, we obtain	
	\begin{equation*}
	\frac{1}{2^{10}} \leq |Dw_4| + s \leq 2^{10} \quad \text{in } B_{2\delta R}.
	\end{equation*}
\end{lemma}
\begin{proof}
Combining the estimates \eqref{duw1}, \eqref{w1w2}, \eqref{w2w3}, \eqref{w3w4}, \eqref{assume1},   \eqref{assume2} and H\"{o}lder's inequality, we arrive at
\begin{align*}
&	(|Dw_4|+s)_{B_{2R}}\ls (|Dw_4 - Dw_3|)_{B_{2R}}\\
	&\quad + 2^n \left[ (|Dw_3 - Dw_2|)_{B_{4R}} + (|Dw_2 - Dw_1|)_{B_{4R}} + (|Dw_1 - Du|)_{B_{4R}}+(|Du|+s)_{B_{4R}} \right] \\
	&\ls \frac{1}{2^{2n+16}} \left[ 2 (|Du|+s)_{B_{4R}} + (|Du - Dw_1|)_{B_{4R}} + (|Dw_1 - Dw_2|)_{B_{4R}}+  (|Dw_2 - Dw_3|)_{B_{4R}} \right] ^{1+\sigma} \\
& \quad+ 2^n \left[ \frac{1}{2^{16}} + \frac{1}{2^{16}} + \frac{1}{2^{16}} + 4 \right] \\
 	&\ls 2^{n+3}.
\end{align*}
Applying Lemma \ref{w4}, we obtain
\begin{equation}\label{210}
	\sup_{x,y \in B_{2\delta R}} |D w_4(x) - D w_4(y)| \leq \frac{1}{2^{10}}.  
\end{equation}
Then we   estimate
\begin{align*}
	\left| (|D u|)_{B_{2\delta R}} - (|D w_4|)_{B_{2\delta R}} \right| &\leq \left| (|D u|)_{B_{2\delta R}} - (|D w_1|)_{B_{2\delta R}} \right| + \left| (|D w_1|)_{B_{2\delta R}} - (|D w_2|)_{B_{2\delta R}} \right| \\
	&\quad + \left| (|D w_2|)_{B_{2\delta R}} - (|D w_3|)_{B_{2\delta R}} \right| + \left| (|D w_3|)_{B_{2\delta R}} - (|D w_4|)_{B_{2\delta R}} \right| \\
	&\leq (2\delta ^{-1})^n \frac{\delta^{2n}}{2^{2n+16}} + 2  (2\delta ^{-1})^n \frac{\delta^{2n}}{2^{2n+16}} + \delta^{-n} \frac{\delta^{2n}}{2^{2n+16}} \left[ 4 + \frac{3}{2^{2n+16}} \right]^{1+\sigma} \\
	&\leq \frac{1}{2^{10}}.
\end{align*}
Combining  \eqref{assume1} with the above inequality to get
\begin{equation*}
	\frac{1}{2^9} \leq (|D w_4| + s)_{B_{2\delta R}} \leq 2^9.
\end{equation*}
Summarizing the above estimates and applying \eqref{210}, we arrive at the final conclusion. This completes the proof.
\end{proof}

To obtain an excess decay estimate in terms of Riesz potentials, we further assume that
\begin{equation}\label{assume3}
	\frac{1}{4} \leqslant (|Du| + s)_{B_{\delta^2 R}} \leqslant 4,
\end{equation}
and that the functions $\bar{w}_1, \bar{w}_2, \bar{w}_3, \bar{w}_4, \bar{w}_5$ satisfy, respectively,
\begin{itemize}
	\item  {for $\bar{w}_1$:}
	\begin{equation*}
		\left\{\begin{array}{r@{\ \ }c@{\ \ }ll}
			&\int_{B_{2\delta R}}  a(x, u, D\bar{w}_1) \cdot D(v-\bar{w}_1)dx \geqslant 0,\ \ &\mbox{for every} \ v \in \bar{w}_1+W_{0}^{1,p}(B_{2\delta R}) \ \  \mbox{with} \ v\geqslant \psi \  a.e. \ \mbox{in}\ \ B_{2\delta R} \,, \\[0.05cm]
			&\bar{w}_1 =u \ \ \ \ \ \ \ \ \ \  \ \ &\mbox{on}\ \ \partial B_{2\delta R} \,;
		\end{array}\right.
	\end{equation*}
	
	\item  {for $\bar{w}_2$:}
	\begin{equation*}
		\left\{\begin{array}{r@{\ \ }c@{\ \ }ll}
			-\operatorname{div} a(x, u, D\bar{w}_2) &=& -\operatorname{div} a(x, u, D\psi)\ \ \ \ \ \  \ &\mbox{in}\ \ B_{2\delta R}\,, \\[0.05cm]
			\bar{w}_2&=&\bar{w}_1 \ \ \ \ \ \ \ \ \ \ \ \ \ \ \  \ \ \ \ \  \ \ \ \ \  &\mbox{on}\ \   \partial B_{2\delta R} \,;
		\end{array}\right.
	\end{equation*}
	
	\item  {for $\bar{w}_3$:}
	\begin{equation*}
		\left\{\begin{array}{r@{\ \ }c@{\ \ }ll}
			-\operatorname{div} a(x, u, D\bar{w}_3) &=& 0\ \ \ \ \ \ \ \ \ \  \ \ \   \ &\mbox{in}\ \ B_{2\delta R}\,, \\[0.05cm]
			\bar{w}_3&=&\bar{w}_2 \ \ \ \ \ \ \ \ \ \  \ \ &\mbox{on}\ \ \partial B_{2\delta R} \,;
		\end{array}\right.
	\end{equation*}
	
	\item  {for $\bar{w}_4$:}
	\begin{equation*}
		\left\{\begin{array}{r@{\ \ }c@{\ \ }ll}
			-\operatorname{div} a(x, (u)_{B_{\delta R}}, D\bar{w}_4)&=& 0\ \ \ \ \ \  \ \ \ \ \   \  \ &\mbox{in}\ \ B_{\delta R}\,, \\[0.05cm]
			\bar{w}_4&=&\bar{w}_3 \ \ \ \ \ \ \ \ \ \  \ \ &\mbox{on}\ \ \partial B_{\delta R} \,;
		\end{array}\right.
	\end{equation*}
	
	\item  {for $\bar{w}_5$:}
	\begin{equation*}
		\left\{\begin{array}{r@{\ \ }c@{\ \ }ll}
			-\operatorname{div} a(x_0, (u)_{B_{\delta R}}, D\bar{w}_5)&=& 0\ \ \ \ \ \  \ \ \ \ \ \ \ \   \ &\mbox{in}\ \ B_{\delta^2 R}\,, \\[0.05cm]
			\bar{w}_5&=&\bar{w}_4 \ \ \ \ \ \ \ \ \ \  \ \ &\mbox{on}\ \ \partial B_{\delta^2 R} \,.
		\end{array}\right.
	\end{equation*}
\end{itemize}

 Define
 \begin{equation*}
 	V(z) := (|z|^2 + s^2)^{\frac{p-2}{4}}   z, \ \ \ \forall z  \in \mathbb{R}^n.
 \end{equation*}
 We observe that
 \begin{equation}\label{vvv}
 	|V(z_1) - V(z_2)|^2 \approx (|z_1|^2 + |z_2|^2 + s^2)^{\frac{p-2}{2}}   |z_1 - z_2|^2, \quad \forall z_1, z_2 \in \mathbb{R}^n.
 \end{equation}
 
 According to the proof of Lemma 5.1 in \cite{bsy} and Lemma 1 in \cite{km155}, we have the following lemma. 
 \begin{lemma}\label{vuw}
 For $h > 0$ and $\xi > 1$, then there exists $c = c(n, p, \ell, L)$ such that
\begin{align*}  
	\left\{
	\begin{aligned}
		&\fint_{B_{2\delta R}} \frac{|V(Du) - V(D\bar{w}_1)|^2}{(h + |u - \bar{w}_1|)^{\xi}} \, dx \leqslant c \frac{h^{1-\xi}}{\xi-1} \frac{|\mu|(B_{2\delta R})}{R^n}, \\
		&\fint_{B_{2\delta R}} \frac{|V(D\bar{w}_1) - V(D\bar{w}_2)|^2}{(h + |\bar{w}_1 - \bar{w}_2|)^{\xi}} \, dx \leqslant c \frac{h^{1-\xi}}{\xi-1} \frac{D\Psi(B_{2\delta R})}{R^n}, \\
		&\fint_{B_{2\delta R}} \frac{|V(D\bar{w}_2) - V(D\bar{w}_3)|^2}{(h + |\bar{w}_2 - \bar{w}_3|)^{\xi}} \, dx \leqslant c \frac{h^{1-\xi}}{\xi-1} \frac{D\Psi(B_{2\delta R})}{R^n}.
	\end{aligned}
	\right.
\end{align*}
 \end{lemma}

\begin{lemma}\label{uw1-}
Suppose that the assumptions \eqref{assume1} and \eqref{assume2} hold. Then we obtain
\begin{equation*}
 \fint_{B_{2\delta R}} |D u - D \bar{w}_1| \, dx \leq c \left[ \frac{|\mu|(B_{4R})}{R^{n-1}} + \frac{D\Psi(B_{4R})}{R^{n-1}} + \Gamma R^{\alpha} \right].   
\end{equation*}
\end{lemma}

\begin{proof}
 Let $$\gamma = \frac{1}{4(p-1)(n+1)}, \ \  \beta = 1 + 2\gamma, \ \  \eta = (p-2)(1 + \gamma).$$
Then, for $p_m$ given in \eqref{pm}, a direct calculation shows that
\begin{equation*}
	1 < \beta < \frac{n}{n-1}  \quad \& \quad p-1 < \eta + 1 < \beta(p-1) < p_m \leq p.
\end{equation*}
Applying the results from Lemma \ref{oscw4}, it follows that
\begin{align}\label{99}  \nonumber
	\fint_{B_{2\delta R}} |D u - D \bar{w}_1| \, dx &\leq c \fint_{B_{2\delta R}} (|D w_4| + s)^\eta |D u - D \bar{w}_1| \, dx \\ \nonumber
	&\leq c \fint_{B_{2\delta R}} |D w_4 - D \bar{w}_1|^\eta |D u - D \bar{w}_1| \, dx + c \fint_{B_{2\delta R}} (|D \bar{w}_1| + s)^\eta |D u - D \bar{w}_1| \, dx \\
	&:= Q_1 + Q_2.  
\end{align}

 Regarding the term $Q_1$, we employ Young's and H\"{o}lder's inequalities. Taking into account the estimates \eqref{duw1}, \eqref{w1w2}, \eqref{w2w3}, \eqref{assume1}, and \eqref{assume2}, we deduce that
\begin{align}\label{91} \nonumber
	Q_1 &\leq c \fint_{B_{2\delta R}} |D w_4 - D \bar{w}_1|^{\eta+1} + |D u - D \bar{w}_1|^{\eta+1} \, dx \\ \nonumber
	&\leq c \fint_{B_{2\delta R}} |D w_4 - D  {w}_3|^{\eta+1} + |D  {w}_3 - D  {w}_2|^{\eta+1} + |D  {w}_2 - D  {w}_1|^{\eta+1} + |D  {w}_1 - D u|^{\eta+1} \\ \nonumber
	&\quad + |D u - D \bar{w}_1|^{\eta+1}+ |D  \bar{w}_1 - D  \bar{w}_2|^{\eta+1} \, dx \\ 
	&\leq c   \Gamma R^\sigma + c   \frac{|\mu|(B_{4R})}{R^{n-1}} + c  \frac{D\Psi(B_{4R})}{R^{n-1}}.  
\end{align}

  Next, we turn to the estimation of $Q_2$.
\begin{align*}
	Q_2 &= \fint_{B_{2\delta R}} \left[ \frac{(|D \bar{w}_1|^2 + s^2)^{\frac{p-2}{2}} |D u - D \bar{w}_1|^2}{(\alpha + |u - \bar{w}_1|)^\beta} \right]^{\frac{1}{2}} \cdot \left[ (|D \bar{w}_1| + s)^{(p-2)\beta} (\alpha + |u - \bar{w}_1|)^\beta \right]^{\frac{1}{2}} \, dx \\
	&\leq \left( \fint_{B_{2\delta R}} \frac{(|D u|^2 + |D \bar{w}_1|^2 + s^2)^{\frac{p-2}{2}} |D u - D \bar{w}_1|^2}{(\alpha + |u - \bar{w}_1|)^\beta} \, dx \right)^{\frac{1}{2}} \\
	&\quad \cdot \left( \fint_{B_{2\delta R}} (|D \bar{w}_1| + s)^{(p-2)\beta} (\alpha + |u - \bar{w}_1|)^\beta \, dx \right)^{\frac{1}{2}} \\
	&\leq c \left[ \alpha^{1-\beta} \frac{|\mu|(B_{4R})}{R^n} \right]^{\frac{1}{2}} \cdot \left( \fint_{B_{2\delta R}} (|D \bar{w}_1| + s)^{(p-2)\beta} (\alpha + |u - \bar{w}_1|)^\beta \, dx \right)^{\frac{1}{2}},
\end{align*}
where we specifically apply \eqref{vvv} and Lemma \ref{vuw} in the final step of the proof.

  Now let $$\displaystyle \alpha= \left( \fint_{B_{2\delta R}} (|D \bar{w}_1| + s)^{(p-2)\beta} |u - \bar{w}_1|^\beta \, dx \right)^{\frac{1}{\beta}}.$$
Combining \eqref{duw1}, \eqref{w1w2}, \eqref{w2w3}, \eqref{w3w4}, \eqref{assume1}, \eqref{assume2} and Lemma \ref{oscw4}, we obtain
\begin{align*}
	&\fint_{B_{2\delta R}} (|D \bar{w}_1| + s)^{(p-2)\beta} (\alpha + |u - \bar{w}_1|)^\beta \, dx \\
	&\leq c \, \alpha^\beta \fint_{B_{2\delta R}} (|D \bar{w}_1| + s)^{(p-2)\beta} \, dx + c\alpha^\beta \\
	&\leq c \alpha^\beta \fint_{B_{2\delta R}}  ( |D \bar{w}_2-D \bar{w}_1| +|D \bar{w}_1 - D u| + |D u - D w_1| + |D w_1 - D w_2| + |D w_2 - D w_3| \\
	&\quad + |D w_3 - D w_4| + |D w_4| + s  )^{(p-2)\beta} + 1 \, dx \\
	&\leq c \, \alpha^\beta.
\end{align*}
Consequently, an application of Young's inequality leads to
\begin{equation*}
	Q_2 \leq c \left[ \frac{\alpha}{R} \right]^{\frac{1}{2}} \cdot \left[ \frac{|\mu|(B_{4R})}{R^{n-1}} \right]^{\frac{1}{2}} \leq \varepsilon \frac{\alpha}{R} + c(\varepsilon) \frac{|\mu|(B_{4R})}{R^{n-1}}.
\end{equation*}
 On the other hand,
\begin{align*}
\alpha &\leq c \left( \fint_{B_{2\delta R}} |D \bar{w}_1 - D w_4|^{(p-2)\beta} |u - \bar{w}_1|^\beta \, dx \right)^{\frac{1}{\beta}} + c \left( \fint_{B_{2\delta R}} (|D w_4| + s)^{(p-2)\beta} |u - \bar{w}_1|^\beta \, dx \right)^{\frac{1}{\beta}} \\
	&:= F_1 + F_2.
\end{align*}

 For $F_1$, by employing Young's inequality, Poincaré's inequality, \eqref{duw1}, \eqref{w1w2},  \eqref{w2w3}, \eqref{w3w4}, \eqref{assume1}  and \eqref{assume2}, we obtain
\begin{align}\label{F1} \nonumber
	\frac{F_1}{R} &\leq c \left( \fint_{B_{2\delta R}} |D \bar{w}_1 - D w_4|^{\beta(p-1)} \, dx \right)^{\frac{1}{\beta}} + c \left( \fint_{B_{2\delta R}} \left| \frac{u - \bar{w}_1}{R} \right|^{\beta(p-1)} \, dx \right)^{\frac{1}{\beta}} \\ \nonumber
	&\leq c \bigg( \fint_{B_{2\delta R}} \Big( |D w_4 - D w_3|^{\beta(p-1)} + |D w_3 - D w_2|^{\beta(p-1)} + |D w_2 - D w_1|^{\beta(p-1)} \\ \nonumber
	&\quad + |D w_1 - D u|^{\beta(p-1)} + |D u - D \bar{w}_1|^{\beta(p-1)}+|D\bar{w}_1 - D \bar{w}_2|^{\beta(p-1)} \Big) \, dx \bigg)^{\frac{1}{\beta}}\\ \nonumber
	&\quad + c \left( \fint_{B_{2\delta R}} |D u - D \bar{w}_1|^{\beta(p-1)} \, dx \right)^{\frac{1}{\beta}} \\
	&\leq  c \frac{|\mu|(B_{4R})}{R^{n-1}} + c \frac{D\Psi(B_{4R})}{R^{n-1}} + c \Gamma R^\sigma.
\end{align} 

 As for $F_2$, in view of Lemma \ref{oscw4} and the Sobolev-type embedding, we obtain
\begin{equation}\label{F2}
	\frac{F_2}{R} \leq c \left( \fint_{B_{2\delta R}} \left| \frac{u - \bar{w}_1}{R} \right|^{\frac{n}{n-1}} \, dx \right)^{\frac{n-1}{n}} \leq c \fint_{B_{2\delta R}} |D u - D \bar{w}_1| \, dx.  
\end{equation}
\noindent Combining \eqref{F1} with \eqref{F2}, we arrive at
\begin{equation*}
	Q_2 \leq c \varepsilon   \fint_{B_{2\delta R}} |D u - D \bar{w}_1| \, dx + c \frac{|\mu|(B_{4R})}{R^{n-1}} + c \frac{D\Psi(B_{4R})}{R^{n-1}} + c \Gamma R^\alpha,
\end{equation*}
which together with \eqref{91} and taking $\varepsilon$ sufficiently small, yields
\begin{equation*}
	\fint_{B_{2\delta R}} |D u - D \bar{w}_1| \, dx \leq c \left[ \frac{|\mu|(B_{4R})}{R^{n-1}} + \frac{D\Psi(B_{4R})}{R^{n-1}} + \Gamma R^\alpha \right].
\end{equation*}
This completes the proof of the lemma.
\end{proof}

\begin{lemma}\label{w1w2-}
 Under the assumptions \eqref{assume1}, \eqref{assume2}, we have
\begin{equation*}
	\fint_{B_{2\delta R}} |D \bar{w}_1 - D \bar{w}_2| \, dx \leq c \left[ \frac{|\mu|(B_{4R})}{R^{n-1}} + \frac{D\Psi(B_{4R})}{R^{n-1}} + \Gamma R^\alpha \right].
\end{equation*}
\end{lemma}
\begin{proof}
  As the proof is analogous to that of Lemma \ref{uw1-}, we shall only indicate the essential ideas. Let $\gamma, \beta, \eta$ be the constants appearing in Lemma \ref{uw1-}. Then the following holds
\begin{align*}
	\fint_{B_{2\delta R}} |D \bar{w}_1 - D \bar{w}_2| \, dx &\leq c \fint_{B_{2\delta R}} |D   {w}_4 - D \bar{w}_2|^\eta |D \bar{w}_1 - D \bar{w}_2| \, dx + c \fint_{B_{2\delta R}} (|D \bar{w}_2| + s)^\eta |D \bar{w}_1 - D \bar{w}_2| \, dx \\
	&:= Q_1 + Q_2.
\end{align*}

 As for the term $Q_1$, we deduce that
\begin{align*}
	Q_1 &\leq c \fint_{B_{2\delta R}} \bigg( |D w_4 - D w_3|^{\eta+1} + |D w_3 - D w_2|^{\eta+1} + |D w_2 - D w_1|^{\eta+1} + |D w_1 - D u|^{\eta+1}\\
	& + |D u - D \bar{w}_1|^{\eta+1} + |D \bar{w}_1 - D \bar{w}_2|^{\eta+1} \bigg) \, dx \\
	&\leq c \Gamma R^\sigma + c \frac{|\mu|(B_{4R})}{R^{n-1}} + c \frac{D\Psi(B_{4R})}{R^{n-1}}.
\end{align*}

Regarding $Q_2$, an application of Lemma \ref{vuw} yields
\begin{align*}
	Q_2 &\leq \left( \fint_{B_{2\delta R}} \frac{(|D \bar{w}_1|^2 + |D \bar{w}_2|^2 + s^2)^{\frac{p-2}{2}} |D \bar{w}_1 - D \bar{w}_2|^2}{(\alpha + |\bar{w}_1 - \bar{w}_2|)^\beta} \, dx \right)^{\frac{1}{2}} \cdot \left( \fint_{B_{2\delta R}} (|D \bar{w}_2| + s)^{(p-2)\beta} (\alpha + |\bar{w}_1 - \bar{w}_2|)^\beta \, dx \right)^{\frac{1}{2}} \\
	&\leq c \left[ \alpha^{1-\beta} \frac{D\Psi(B_{4R})}{R^n} \right]^{\frac{1}{2}} \cdot  \left( \fint_{B_{2\delta R}} (|D \bar{w}_2| + s)^{(p-2)\beta} (\alpha + |\bar{w}_1 - \bar{w}_2|)^\beta \, dx \right)^{\frac{1}{2}}.
\end{align*}

Next, we let $$\alpha = \left( \fint_{B_{2\delta R}} (|D \bar{w}_2| + s)^{(p-2)\beta} |\bar{w}_1 - \bar{w}_2|^\beta \, dx \right)^{\frac{1}{\beta}},$$  
then we have
\begin{align*}
&	\fint_{B_{2\delta R}} (|D \bar{w}_2| + s)^{(p-2)\beta} (\alpha + |\bar{w}_1 - \bar{w}_2|)^\beta dx \\
	  &\leq c \alpha^\beta \fint_{B_{2\delta R}} \Big( |D\bar{w}_2 - D \bar{w}_1| + |D\bar{w}_1 - D u|  + |D u - D w_1|  \\  
	&\quad + |D w_1 - D  w_2|  + |D  w_2 - D  w_3| +|D  w_3 - D  w_4|+|Dw_4|+s \Big)^{\beta(p-2)}+1 \, dx \\
	&\leq c \alpha^\beta.
\end{align*}
Therefore, 
\begin{equation*}
	Q_2  \leq \varepsilon \frac{\alpha}{R} + c(\varepsilon) \frac{D\Psi(B_{4R})}{R^{n-1}}.
\end{equation*}
Next we estimate $\alpha$,
\begin{align*}
\alpha &\leq c \left( \fint_{B_{2\delta R}} |D \bar{w}_2 - D w_4|^{(p-2)\beta} |\bar{w}_1 - \bar{w}_2|^\beta \, dx \right)^{\frac{1}{\beta}} + c \left( \fint_{B_{2\delta R}} (|D w_4| + s)^{(p-2)\beta} |\bar{w}_1 - \bar{w}_2|^\beta \, dx \right)^{\frac{1}{\beta}} \\
	&:= F_1 + F_2.
\end{align*}

Concerning the term $F_1$, we obtain  
\begin{align*}
	\frac{F_1}{R} &\leq c \bigg( \fint_{B_{2\delta R}} |D w_4 - D w_3|^{\beta(p-1)} + |D w_3 - D w_2|^{\beta(p-1)} + |D w_2 - D w_1|^{\beta(p-1)} \\
	&\quad + |D w_1 - D u|^{\beta(p-1)} + |D u - D \bar{w}_1|^{\beta(p-1)}+ |D \bar{w}_1 - D \bar{w}_2|^{\beta(p-1)}\, dx \bigg)^{\frac{1}{\beta}} \\
	&\quad+ c \left( \fint_{B_{2\delta R}} |D \bar{w}_1 - D \bar{w}_2|^{\beta(p-1)} \, dx \right)^{\frac{1}{\beta}} \\
	 &\leq c\frac{|\mu|(B_{4R})}{R^{n-1}} + c \frac{D\Psi(B_{4R})}{R^{n-1}} + c \Gamma R^\alpha.
\end{align*}

As for the term $F_2$, we derive
\[
\frac{F_2}{R} \le c \fint_{B_{2\delta R}} |D\bar{w}_1 - D\bar{w}_2| \, dx.
\]
Consequently, we have
\[
Q_2 \le c \varepsilon    \fint_{B_{2\delta R}} |D\bar{w}_1 - D\bar{w}_1| \, dx + c \frac{|\mu|(B_{4R})}{R^{n-1}} + c \frac{ {D}\Psi(B_{4R})}{R^{n-1}} + c \Gamma R^\alpha.
\]
Combining the previous estimates for $Q_1$ and $Q_2$, and choosing $\varepsilon$ to be sufficiently small to absorb the corresponding terms, we have
\[
\fint_{B_{2\delta R}} |D\bar{w}_1 - D\bar{w}_2| \, dx \le c \left[ \frac{|\mu|(B_{4R})}{R^{n-1}} + \frac{ {D}\Psi(B_{4R})}{R^{n-1}} + \Gamma R^\alpha \right].
\]
This completes the proof.
\end{proof}
The proof strategy of the following lemma is similar to Lemma \ref{uw1-} and Lemma \ref{w1w2-}, with the key distinction being the utilization of Lemma \ref{vuw} in the proof process.

\begin{lemma}\label{w2w3-}
	Suppose that the assumptions \eqref{assume1} and \eqref{assume2} hold. Then we have
	\[
	\fint_{B_{2\delta R}} |D\bar{w}_2 - D\bar{w}_3| \, dx \le c \left[ \frac{|\mu|(B_{4R})}{R^{n-1}} + \frac{ {D}\Psi(B_{4R})}{R^{n-1}} + \Gamma R^\alpha \right].
	\]
\end{lemma}

\begin{lemma}\label{w3w4-}
	Under the assumptions \eqref{assume1} and \eqref{assume2}, we derive
	\[
	\fint_{B_{\delta R}} |D\bar{w}_3 - D\bar{w}_4| \, dx \le c \Gamma R^\alpha.
	\]
\end{lemma}

\begin{proof}
 
A combination of \eqref{duw1}, \eqref{w1w2}, \eqref{w2w3}, \eqref{assume1} and \eqref{assume2},  yields
\begin{align*}
	\fint_{B_{\delta R}} |D\bar{w}_3 - D\bar{w}_4| \, dx 
	&\le c \Gamma R^\alpha   \left( \fint_{B_{2\delta R}}( |D {u} | + |D\bar{w}_3| + s) \, dx \right)^{1+\sigma } \\
	&\le c \Gamma R^\alpha \cdot \left[ 2(|D {u} | + s)_{B_{2\delta R}} + (|D {u} - D\bar{w}_1|)_{B_{2\delta R}} + (|D\bar{w}_1 - D\bar{w}_2|)_{B_{2\delta R}} \right. \\
	&\quad \left. + (|D\bar{w}_2 - D\bar{w}_3|)_{B_{2\delta R}} \right]^{1+\sigma} \\
	&\le c \Gamma R^\alpha,
\end{align*}
which completes the proof.
\end{proof}

\begin{lemma}  \label{w4w5-}
Assume that the conditions \eqref{assume1}, \eqref{assume2}  and \eqref{assume3} are satisfied. Then we obtain
	\begin{align*}
		\fint_{B_{\delta^2 R}} |D\bar{w}_4 - D\bar{w}_5| \, dx \le c \omega(R).
	\end{align*}
\end{lemma}
\begin{proof}
By virtue of \eqref{duw1}, \eqref{w1w2}, \eqref{w2w3}, \eqref{w3w4}, \eqref{assume1}  and \eqref{assume2}, we arrive at
	\begin{align*}
		(|D\bar{w}_4| + s)_{B_{\delta R}} 
		&\le (|D\bar{w}_4 - D\bar{w}_3|)_{B_{\delta R}} + 2^n \left[ (|D\bar{w}_3 - D\bar{w}_2|)_{B_{2\delta R}} \right. \nonumber \\
		&\quad \left. + (|D\bar{w}_2 - D\bar{w}_1|)_{B_{2\delta R}} + (|D\bar{w}_1 - D {u}|)_{B_{2\delta R}} + (|Du| + s)_{B_{2\delta R}} \right] \nonumber \\
		&\le \frac{\delta^{2n}}{2^{2n+16}} \left[ 2(|Du| + s)_{B_{2\delta R}} + (|Du - D\bar{w}_1|)_{B_{2\delta R}} + (|D\bar{w}_1 - D\bar{w}_2|)_{B_{2\delta R}} \right. \nonumber \\
		&\quad \left. + (|D\bar{w}_2 - D\bar{w}_3|)_{B_{2\delta R}} \right]^{1+\sigma} \nonumber \\
		&\quad + 2^n \left\{ c_1 \left[ \frac{|\mu|(B_{2\delta R})}{(2\delta R)^{n-1}} \right]^{\frac{1}{p-1}} + 2c_2 \left[ \frac{ {D}\Psi(B_{2\delta R})}{(2\delta R)^{n-1}} \right]^{\frac{1}{p-1}} + 4 \right\} \nonumber \\
		&\le \frac{\delta^{2n}}{2^{2n+16}}   \left[ 8 + \frac{3 \cdot (2\delta^{-1})^{\frac{n-1}{p-1}} \cdot \delta^{2n}}{2^{2n+16}} \right]^{1+\sigma} \nonumber \\
		&\quad + 2^n   \left[ \frac{3 \cdot (2\delta^{-1})^{\frac{n-1}{p-1}} \cdot \delta^{2n}}{2^{2n+16}} + 4 \right] \nonumber \\
		&\le 2^{n+3}.
	\end{align*}
Then it  follows from Lemma \ref{w4} that
	\begin{align}\label{dw4e}
		\sup_{x,y \in B_{\delta^2 R}} |D \bar{w}_4(x) - D \bar{w}_4(y)| \le \frac{1}{2^{10}}.  
	\end{align}
	On the other hand, we estimate
	\begin{align*}
		\left| (|Du|)_{B_{\delta^2 R}} - (|D\bar{w}_4|)_{B_{\delta^2 R}} \right| 
		&\le \left| (|Du|)_{B_{\delta^2 R}} - (|D\bar{w}_1|)_{B_{\delta^2 R}} \right| + \left| (|D\bar{w}_1|)_{B_{\delta^2 R}} - (|D\bar{w}_2|)_{B_{\delta^2 R}} \right| \nonumber \\
		&\quad + \left| (|D\bar{w}_2|)_{B_{\delta^2 R}} - (|D\bar{w}_3|)_{B_{\delta^2 R}} \right| + \left| (|D\bar{w}_3|)_{B_{\delta^2 R}} - (|D\bar{w}_4|)_{B_{\delta^2 R}} \right| \nonumber \\
		&\le 3   (2\delta^{-1})^n   \frac{(2\delta^{-1})^{\frac{n-1}{p-1}}\delta^{2n}}{2^{2n+16}} + \delta^{-n}   \frac{ \delta^{2n}}{2^{2n+16}}   \left[ 4 + \frac{3   (2\delta^{-1})^{\frac{n-1}{p-1}}   \delta^{2n}}{2^{2n+16}} \right]^{1+\sigma} \nonumber \\
		&\le \frac{1}{2^{10}}.
	\end{align*}
From \eqref{assume3}, it follows that
	\begin{align*}
		\frac{1}{2^9} \le (|D\bar{w}_4| + s)_{B_{\delta^2 R}} \le 2^9.
	\end{align*}
Then \eqref{dw4e} leads to the estimate
	\begin{align*}
		\frac{1}{2^{10}} \le |D\bar{w}_4| + s \le 2^{10} \quad \text{in} \quad B_{\delta^2 R}.
	\end{align*}
Noting that $p\geqslant2$ and applying \eqref{w4w5---}, we obtain
\begin{align*}
	\fint_{B_{\delta^2 R}} |D\bar{w}_4 - D\bar{w}_5|^2 \, dx 
		&\le   c\fint_{B_{\delta^2 R}} (|D\bar{w}_4|^2  + s^2)^{\frac{p-2}{2}} \cdot |D\bar{w}_4 - D\bar{w}_5|^2 \, dx   \nonumber \\
	&\le   c\fint_{B_{\delta^2 R}} (|D\bar{w}_4|^2 + |D\bar{w}_5|^2 + s^2)^{\frac{p-2}{2}} \cdot |D\bar{w}_4 - D\bar{w}_5|^2 \, dx   \nonumber \\
	&\le c   w(R)^2 \fint_{B_{\delta^2 R}} (|D\bar{w}_4| + s)^p \, dx  \nonumber \\
	&\le c\omega(R)^2.
\end{align*}
Finally, the proof of the lemma is completed by applying H\"older's inequality.
\end{proof}

\begin{lemma}\label{oscu}
	Under the assumptions \eqref{assume1}, \eqref{assume2} and \eqref{assume3},  we obtain
	\begin{align*}
		\fint_{B_{\frac{\delta^3 R}{2}}} |Du - (Du)_{B_{2\delta^3 R}}| \, dx 
		&\le \frac{1}{4} \fint_{B_{\delta^2 R}} |Du - (Du)_{B_{\delta^2 R}}| \, dx  \nonumber \\
		&\quad +c_7 \left[ \frac{|\mu|(B_{4R})}{(4R)^{n-1}} + \frac{ {D}\Psi(B_{4R})}{(4R)^{n-1}} + \Gamma (4R)^\alpha + \omega(4R) \right].
	\end{align*}
\end{lemma}

\begin{proof}
	By virtue of Lemma \ref{w5}, we have
	\begin{align*}
		\fint_{B_{2\delta^3 R}} |Du - (Du)_{B_{2\delta^3 R}}| \, dx 
		&\le \fint_{B_{2\delta^3 R}} |D\bar{w}_5 - (D\bar{w}_5)_{B_{2\delta^3 R}}| \, dx + \fint_{B_{2\delta^3 R}} |Du - D\bar{w}_5| \, dx \nonumber \\
		&\le \frac{1}{2^{2n+10}} \fint_{B_{\delta^2 R}} |D\bar{w}_5 - (D\bar{w}_5)_{B_{\delta^2 R}}| \, dx + \fint_{B_{2\delta^3 R}} |Du - D\bar{w}_5| \, dx \nonumber \\
		&\le \frac{1}{2^{2n+10}} \fint_{B_{\delta^2 R}} |Du - (Du)_{B_{\delta^2 R}}| \, dx + c \fint_{B_{\delta^2 R}} |Du - D\bar{w}_5| \, dx.
	\end{align*}
	We make use of Lemma \ref{uw1-}, Lemma \ref{w1w2-}, Lemma \ref{w2w3-}, Lemma \ref{w3w4-}, and Lemma \ref{w4w5-} to conclude that
	\begin{align*}
		\fint_{B_{\delta^2 R}} |Du - D\bar{u}_5| \, dx 
		&\le c \fint_{B_{2\delta  R}} |Du - D\bar{w}_1|+  |D\bar{w}_1 - D\bar{w}_2|+|D\bar{w}_2 - D\bar{w}_3|\,   dx   \\
			&\le  c \fint_{B_{\delta R}} |D\bar{w}_3 - D\bar{w}_4| \, dx + \fint_{B_{\delta^2 R}} |D\bar{w}_4 - D\bar{w}_5| \, dx \\
		&\le c \left[ \frac{|\mu|(B_{4R})}{R^{n-1}} + \frac{ {D}\Psi(B_{4R})}{R^{n-1}} + \Gamma R^\alpha + \omega(R) \right].
\end{align*}
Combining the above inequalities, we complete the proof of the lemma.
\end{proof}

\section{The proof of main theorem} 
This section is devoted to the proofs of the main theorems.

Let $\delta \in (0, \frac{1}{8}]$, $B_{3r_0}(x_0)\subseteq \Omega$ and $c_1, c_2, c_3, c_4, c_5, c_6, c_7$ be the universal constants appearing in \eqref{delta}, \eqref{duw1}, \eqref{w1w2}, \eqref{w3w4}, \eqref{w4w5}, Lemma \ref{w4}, Lemma \ref{du-6}, and Lemma \ref{oscu}, respectively. Let
\begin{align}\label{constant} 
	\varepsilon = \frac{\delta}{2} \in (0, \frac{1}{16}], \quad c_8 = c_1+c_2+c_3+c_4+c_5+c_6+c_7, \quad r_i = \varepsilon^i r_0, \quad B_i = B_{r_i}(x_0). 
\end{align}
Assume that
\begin{align}\label{c8p}
	c_8^p \left[ \int_{0}^{3r_0} \frac{|\mu|(B_\rho)}{\rho^{n-1}} \frac{d\rho}{\rho} + \int_{0}^{3r_0} \frac{ {D}\Psi(B_\rho)}{\rho^{n-1}} \frac{d\rho}{\rho} + \frac{\Gamma r_0^\sigma}{1-(\frac{1}{2})^\sigma} + \int_{0}^{3r_0} \frac{\omega(\rho)}{\rho} d\rho \right] \le \left( \frac{\varepsilon^{3n}}{2^{19}} \right)^p. 
\end{align}
A direct calculation yields
\begin{align*}
	&\sum_{i=1}^{\infty} \frac{|\mu|(B_i)}{r_i^{n-1}} + \frac{ D\Psi(B_i)}{r_i^{n-1}} + \sum_{i=0}^{\infty} r_i^\sigma + \sum_{i=1}^{\infty} \omega(r_i) \nonumber \\
	&\le \frac{1}{\varepsilon^{n-1}(-\log \varepsilon)} \left[ \int_{0}^{r_0} \frac{|\mu|(B_\rho)}{\rho^{n-1}} \frac{d\rho}{\rho} + \int_{0}^{r_0} \frac{D\Psi(B_\rho)}{\rho^{n-1}} \frac{d\rho}{\rho} \right] + \frac{r_0^\sigma}{1-\varepsilon^\sigma} + \frac{1}{(-\log \varepsilon)} \int_{0}^{r_0} \frac{\omega(\rho)}{\rho} d\rho
\end{align*}
and
\begin{align*}
	\frac{|\mu|(B_0)}{r_0^{n-1}} + \frac{D\Psi(B_0)}{r_0^{n-1}} + \omega(r_0) \le \frac{3^{n-1}}{\log 3} \left[ \int_{r_0}^{3r_0} \frac{|\mu|(B_\rho)}{\rho^{n-1}} \frac{d\rho}{\rho} + \int_{r_0}^{3r_0} \frac{D\Psi(B_\rho)}{\rho^{n-1}} \frac{d\rho}{\rho} \right] + \frac{1}{\log 3} \int_{r_0}^{3r_0} \frac{\omega(\rho)}{\rho} d\rho.
\end{align*}
It follows from \eqref{c8p} and the above inequalities that
\begin{align}\label{c8p2}
	c_8^p \sum_{i=0}^{\infty} \left( \frac{|\mu|(B_i)}{r_i^{n-1}} + \frac{D\Psi(B_i)}{r_i^{n-1}} + \Gamma r_i^\sigma + \omega(r_i) \right) \le \left( \frac{\varepsilon^{3n}}{2^{19}} \right)^p 
\end{align}
and
\begin{align}\label{c8p3}
	c_1 \left[ \frac{|\mu|(B_i)}{r_i^{n-1}} \right]^{\frac{1}{p-1}} + c_2 \left[ \frac{D\Psi(B_i)}{r_i^{n-1}} \right]^{\frac{1}{p-1}} + c_3 \Gamma r_i^\sigma + \omega(r_i) + c_5 \int_{0}^{r_i} \frac{\omega(\rho)}{\rho} d\rho \le \frac{5 \varepsilon^{2n}}{2^{19}} \le \frac{\delta^{2n}}{2^{2n+16}}, \quad i \ge 0. 
\end{align}

As a first step in the inductive argument, we prove the following lemma.
\begin{lemma}\label{ddla}
	Under the assumption \eqref{c8p}, suppose that
	\begin{align}\label{ddla-1}
		(|Du| + s)_{B_0} \le 2, \quad \frac{1}{2} \le (|Du| + s)_{B_1} \le 3  
	\end{align}
	\begin{align}\label{ddla-2}
		(|Du - (Du)_{B_1}|)_{B_1} \le \frac{\varepsilon^n}{2^{14}}, \quad (|Du - (Du)_{B_2}|)_{B_2} \le \frac{\varepsilon^n}{2^{14}}  
	\end{align}
	and
	\begin{align}\label{ddla-3}
		\frac{1}{4} \le (|Du| + s)_{B_j} \le 4, \quad j \in [2, i], \quad i\rs 2.  
	\end{align}
	Then  
	\begin{align*}
		\frac{1}{4} \le (|Du| + s)_{B_{i+1}} \le 4.
	\end{align*}
\end{lemma}

\begin{proof}
  By \eqref{c8p3}, \eqref{ddla-1} and \eqref{ddla-3}, the assumption of Lemma \ref{oscu} is satisfied for $j \in [2, i]$, and hence
  \begin{align}\label{bj1}
  	\fint_{B_{j+1}} |Du - (Du)_{B_{j+1}}| \, dx 
  	&\le \frac{1}{4} \fint_{B_j} |Du - (Du)_{B_j}| \, dx \nonumber \\
  	&\quad + c_7 \left[ \frac{|\mu|(B_{j-2})}{r_{j-2}^{n-1}} + \frac{ {D}\Psi(B_{j-2})}{r_{j-2}^{n-1}} + \Gamma r_{j-2}^\sigma + \omega(r_{j-2}) \right].
  \end{align}
 Summing \eqref{bj1} from $j=2$ to $i$ and using \eqref{constant}, \eqref{c8p2}, and \eqref{ddla-2}, we obtain
  \begin{align*}
  	\sum_{j=3}^{i+1} \fint_{B_j} |Du - (Du)_{B_j}| \, dx 
  	&\le \fint_{B_2} |Du - (Du)_{B_2}| \, dx \nonumber \\
  	&\quad + 2c_7 \sum_{j=0}^{i} \left[ \frac{|\mu|(B_j)}{r_j^{n-1}} + \frac{{D}\Psi(B_j)}{r_j^{n-1}} + \Gamma r_j^\sigma + \omega(r_j) \right] \nonumber \\
  	&\le \frac{\varepsilon^n}{2^{10}}.
  \end{align*}
  It follows from \eqref{ddla-2} and the above inequality that
  \begin{align*}
  	\sum_{j=1}^{i+1} \fint_{B_j} |Du - (Du)_{B_j}| \, dx \le \frac{\varepsilon^n}{2^8}.
  \end{align*}
 By virtue of Lemma 2.4 in \cite{kim}, we obtain
  \begin{align*}
  	\left| (Du)_{B_{i+1}} - (Du)_{B_1} \right| &\le \left( |Du - (Du)_{B_{i+1}}| \right)_{B_{i+1}} + \left| (Du)_{B_{i+1}} - (Du)_{B_1} \right| + \left( |Du - (Du)_{B_1}| \right)_{B_1} \nonumber \\
  	&\le 2 \varepsilon^{-n} \sum_{j=1}^{i+1} \left( |Du - (Du)_{B_j}| \right)_{B_j} \nonumber \\
  	&\le \frac{1}{2^7}.
  \end{align*}
 Finally, by \eqref{ddla-1}, we conclude the proof.
 \end{proof}
  
  \begin{lemma}\label{b0b1}
  	Under the assumption \eqref{c8p}, suppose that
  	\begin{align}\label{b0b1-1}
  		\begin{cases}
  			(|Du| + s)_{B_0} + 2^9 \varepsilon^{-n} (|Du - (Du)_{B_0}|)_{B_0} \ls 2 \\
  			(|Du| + s)_{B_1} + 2^9 \varepsilon^{-n} (|Du - (Du)_{B_1}|)_{B_1} > 1.
  		\end{cases}  
  	\end{align}
  	Then  
  	\begin{align}\label{4i1}
  		\frac{1}{4} \ls (|Du| + s)_{B_i} \ls 4, \quad i \rs 1
  	\end{align}
  	and
   \begin{align}\label{4i2}
  	\fint_{B_{i+1}} |Du - (Du)_{B_{i+1}}| \, dx 
  	&\ls \frac{1}{4} \fint_{B_i} |Du - (Du)_{B_i}| \, dx \nonumber \\
  	&\quad + c_7 \left[ \frac{|\mu|(B_{i-2})}{r_{i-2}^{n-1}} + \frac{ {D}\Psi(B_{i-2})}{r_{i-2}^{n-1}} + \Gamma r_{i-2}^\sigma + \omega(r_{i-2}) \right], \ \ i\rs2.
  \end{align}
  \end{lemma}

\begin{proof}
Lemma 2.4 in \cite{kim} yields
	\begin{align}\label{bii1}
		&\left| (|Du|)_{B_{i+1}} - (|Du|)_{B_i} \right| \nonumber \\
		&\le (|Du - (Du)_{B_{i+1}}|)_{B_{i+1}} + |(Du)_{B_{i+1}} - (Du)_{B_i}| + (|Du - (Du)_{B_i}|)_{B_i} \nonumber \\
		&\le (|Du - (Du)_{B_{i+1}}|)_{B_{i+1}} + [1 + \varepsilon^{-n}] (|Du - (Du)_{B_i}|)_{B_i}.  
	\end{align}
	By \eqref{c8p2} and \eqref{b0b1-1}, the assumption of Lemma \ref{du-6} is satisfied, and we obtain
	\begin{align}\label{js1}
		\fint_{B_1} |Du - (Du)_{B_1}| \, dx &\le \frac{1}{2^7} \fint_{B_0} |Du - (Du)_{B_0}| \, dx \nonumber \\
		&\quad + c_6\left\lbrace  \left[ \frac{|\mu|(B_0)}{r_0^{n-1}} \right]^{\frac{1}{p-1}} + \left[ \frac{ {D}\Psi(B_0)}{r_0^{n-1}} \right]^{\frac{1}{p-1}} + \Gamma r_0^\sigma + \omega(r_0)^{\frac{2}{p}}\right\rbrace  \nonumber \\
		&\le \frac{1}{2^7}   \frac{\varepsilon^n}{2^9} + \frac{4   \varepsilon^n}{2^{19}} \nonumber \\
		&\le \frac{\varepsilon^n}{2^{14}}, 
	\end{align}
where \eqref{c8p2} and \eqref{b0b1-1} have been used in the last step.
	
	Next,  \eqref{b0b1-1}, \eqref{bii1}   and \eqref{js1} imply
	\begin{align*}
		\left| (|Du|)_{B_1} - (|Du|)_{B_0} \right| \le (|Du - (Du)_{B_1}|)_{B_1} + [1 + \varepsilon^{-n}] (|Du - (Du)_{B_0}|)_{B_0} \le \frac{1}{2^6}.
	\end{align*}
	Therefore,
	\begin{align}\label{js2}
		\begin{cases} (|Du| + s)_{B_1} \le | (|Du|)_{B_1} - (|Du|)_{B_0} | + (|Du| + s)_{B_0} \le 3, \\ (|Du| + s)_{B_1} > 1 - 2^9 \varepsilon^{-n} (|Du - (Du)_{B_1}|)_{B_1} \ge \frac{1}{2}. \end{cases}  
	\end{align}
	Then we apply Lemma \ref{du-6} again to obtain
	\begin{align}\label{js3}
		\fint_{B_2} |Du - (Du)_{B_2}| \, dx &\le \frac{1}{2^7} \fint_{B_1} |Du - (Du)_{B_1}| \, dx \nonumber \\
		&\quad + c_6 \left[ \left( \frac{|\mu|(B_1)}{r_1^{n-1}} \right)^{\frac{1}{p-1}} + \left( \frac{ {D}\Psi(B_1)}{r_1^{n-1}} \right)^{\frac{1}{p-1}} + \Gamma r_1^\sigma + \omega(r_1)^{\frac{2}{p}} \right] \nonumber \\
		&\le \frac{1}{2^7}   \frac{\varepsilon^n}{2^{14}} + \frac{4   \varepsilon^n}{2^{19}} \nonumber \\
		&\le \frac{\varepsilon^n}{2^{14}}, 
	\end{align}
	where \eqref{c8p2} and \eqref{js1} have been used in the last step. By \eqref{bii1}, \eqref{js3}, and \eqref{js1}, we obtain
	\begin{align*}
		| (|Du|)_{B_2} - (|Du|)_{B_1} | \le (|Du - (Du)_{B_2}|)_{B_2} + [1 + \varepsilon^{-n}] (|Du - (Du)_{B_1}|)_{B_1} \le \frac{1}{2^{10}}.
	\end{align*}
Thus, from \eqref{js2} we obtain
	\begin{align}\label{js4}
		\frac{1}{4} \le (|Du| + s)_{B_2} \le 4. 
	\end{align}
Combining \eqref{js1}, \eqref{js2}, \eqref{js3} with \eqref{js4}, we see that the assumption of Lemma \ref{ddla} is satisfied. By an induction argument using Lemma \ref{ddla}, we obtain \eqref{4i1}, and then we apply Lemma \ref{oscu} to obtain \eqref{4i2}. This completes the proof.
\end{proof}

In preparation for the proof of Theorem \ref{thm}, we first establish the following   lemma.
\begin{lemma}\label{thmM1}
	Under the assumptions \eqref{0a}, \eqref{omegaa} and \eqref{omegaz}, let $B_{3r}(x_0)\subseteq \Omega$, $B_{3r}:=B_{3r}(x_0)$, 
	\begin{align}\label{mp-1}
		M^{p-1} = c_8^p \left( \frac{\varepsilon^{3n}}{2^{30}} \right)^{-p} \left[ \int_{0}^{3r} \frac{|\mu|(B_\rho)}{\rho^{n-1}} \frac{d\rho}{\rho} + \int_{0}^{3r} \frac{ {D}\Psi(B_\rho)}{\rho^{n-1}} \frac{d\rho}{\rho} + \left( \fint_{B_{3r}} (|Du|+s) \, dx \right)^{p-1} \right] 
	\end{align}
	and $u \in W^{1,1}(B_{3r})$ with $u \rs \psi$ a.e. be a limit of approximating solutions to $OP(\psi; \mu)$. Assume that
	\begin{align}\label{2300}
		c_8^p \int_{0}^{3r  } \frac{\omega(\rho)}{\rho} d\rho \le \left( \frac{\varepsilon^{3n}}{2^{30}} \right)^p \quad \& \quad \frac{c_8^p \Gamma r ^\sigma M^{\sigma}}{1 - (\frac{1}{2})^\sigma} \le \left( \frac{\varepsilon^{3n}}{2^{30}} \right)^p. 
	\end{align}
	If $x_0$ is a Lebesgue point of $Du$, then 
	\begin{align*}
		|Du(x_0)| + s \le 8M.
	\end{align*}
\end{lemma}

\begin{proof}
	We first normalize the problem by means of \eqref{mp-1}
	\begin{align}\label{normali}
		\tilde{a}(x, z, \eta) = \frac{a(x, Mz, M\eta)}{M^{p-1}}, \quad \tilde{u}(x) = \frac{u(x)}{M}, \quad \tilde{v}(x) = \frac{v(x)}{M}, \quad \tilde{\psi}(x) = \frac{\psi(x)}{M}, \quad \tilde{\mu} = \frac{\mu}{M^{p-1}}.  
	\end{align}
A direct calculation then yields
	\begin{align}
		\int_{B_{3r}} \tilde{a}(x, \tilde{u}, D\tilde{u}) \cdot D(\tilde{v} - \tilde{u}) \, dx \ge \int_{B_{3r}} (\tilde{v} - \tilde{u}) \, d\tilde{\mu}
	\end{align}
	for every $\tilde{v} \in \tilde{u} + W_0^{1,p}(B_{3r})$ with $\tilde{v} \ge \tilde{\psi}$ a.e. in $B_{3r}$.
	And $\tilde{a}$ satisfies
	\begin{align*}
		\begin{cases}
			|\tilde{a}(x, z, \eta)| + |D_\eta \tilde{a}(x, z, \eta)| (\tilde{s}^2 + |\eta|^2)^{\frac{1}{2}} \le L (\tilde{s}^2 + |\eta|^2)^{\frac{p-1}{2}} \\
			D_\eta \tilde{a}(x, z, \eta) \cdot \xi \cdot \xi \ge l (\tilde{s}^2 + |\eta|^2)^{\frac{p-2}{2}} |\xi|^2 \\
			|\tilde{a}(x, z, \eta) - \tilde{a}(x_0, z, \eta)| \le 2L \omega(|x - x_0|) (\tilde{s}^2 + |\eta|^2)^{\frac{p-1}{2}} \\
			|\tilde{a}(x, z_1, \eta) - \tilde{a}(x, z_2, \eta)| \le \tilde{\Gamma}^{p-1} |z_1 - z_2|^{\sigma (p-1)} (|\eta|^2 + \tilde{s}^2)^{\frac{p-1}{2}}
		\end{cases}
	\end{align*}
	for every $x, x_0, \xi, \eta \in \mathbb{R}^n, z_1, z_2, z \in \mathbb{R}$ and for some constants $\tilde{\Gamma} = \Gamma M^\sigma$ and $\tilde{s} = s M^{-1}$. Moreover, we obtain from \eqref{mp-1} and \eqref{normali} that
	\begin{align}\label{2301}
		c_8^p \left( \frac{\varepsilon^{3n}}{2^{30}} \right)^{-p} \left[ \int_{0}^{3r} \frac{|\tilde{\mu}|(B_\rho)}{\rho^{n-1}} \frac{d\rho}{\rho} + \int_{0}^{3r} \frac{ {D}\tilde{\Psi}(B_\rho)}{\rho^{n-1}} \frac{d\rho}{\rho} + \left( \fint_{B_{3r}} (|D\tilde{u}| + \tilde{s}) \, dx \right)^{p-1} \right] = 1  
	\end{align}
	and
	\begin{align}\label{2302}
		\frac{c_8^p \tilde{\Gamma} r ^\sigma}{1 - (\frac{1}{2})^\sigma} = \frac{c_8^p \Gamma r ^\sigma M^\sigma}{1 - (\frac{1}{2})^\sigma} \le \left( \frac{\varepsilon^{3n}}{2^{30}} \right)^p,
	\end{align}
	where $ {D}\tilde{\Psi}(B_\rho) = \int_{B_\rho} |\text{div } \tilde{a}(x, \tilde{u}, D\tilde{\psi})| \, dx$.
	Since $\varepsilon \in (0, 1/16]$, it follows from \eqref{2301} that
	\begin{align}\label{2303}
		\fint_{B_r} (|D\tilde{u}| + \tilde{s}) \, dx \le 3^n \fint_{B_{3r}} (|D\tilde{u}| + \tilde{s}) \, dx \le \frac{3^n \varepsilon^{3n}}{2^{30}} \le \frac{\varepsilon^n}{2^{11}}.  
	\end{align}
Setting $r_i = \varepsilon^i r$ and $B_i = B_{r_i}$ for $i \ge 0$, it follows from \eqref{2303} that
	\begin{align*}
		(|D\tilde{u}| + \tilde{s})_{B_0} + 2^9 \varepsilon^{-n} (|D\tilde{u} - (D\tilde{u})_{B_0}|)_{B_0} \le 2^{11} \varepsilon^{-n} (|D\tilde{u}| + \tilde{s})_{B_0} \le 1.
	\end{align*}
	
	If $(|D\tilde{u}| + \tilde{s})_{B_i} + 2^9 \varepsilon^{-n} (|D\tilde{u} - (D\tilde{u})_{B_i}|)_{B_i} \le 1$ for every $i \ge 0$, then \eqref{normali} implies that $(|Du| + s)_{B_i} \le M$ for every $i \ge 0$. The lemma then follows from the Lebesgue differentiation theorem. Suppose to the contrary that there exists $j \ge 0$ such that
	\begin{align*}
		\begin{cases} (|D\tilde{u}| + \tilde{s})_{B_j} + 2^9 \varepsilon^{-n} (|D\tilde{u} - (D\tilde{u})_{B_j}|)_{B_j} \le 1, \\ (|D\tilde{u}| + \tilde{s})_{B_{j+1}} + 2^9 \varepsilon^{-n} (|D\tilde{u} - (D\tilde{u})_{B_{j+1}}|)_{B_{j+1}} > 1. \end{cases}
	\end{align*}
Let $\tilde{r}_i = \varepsilon^i r_j$ and $\tilde{B}_i = B_{\tilde{r}_i}$. Then we obtain
	\begin{align}\label{2304}
		\begin{cases} (|D\tilde{u}| + \tilde{s})_{\tilde{B}_0} + 2^9 \varepsilon^{-n} (|D\tilde{u} - (D\tilde{u})_{\tilde{B}_0}|)_{\tilde{B}_0} \le 1, \\ (|D\tilde{u}| + \tilde{s})_{\tilde{B}_1} + 2^9 \varepsilon^{-n} (|D\tilde{u} - (D\tilde{u})_{\tilde{B}_1}|)_{\tilde{B}_1} > 1. \end{cases}  
	\end{align}
	
From \eqref{2300}, \eqref{2301}, \eqref{2302}, and \eqref{2304}, we find that $\tilde{u}, \tilde{s}, \tilde{\mu}, \omega$, and $\tilde{\Gamma}$ satisfy the assumption of Lemma \ref{b0b1} with $u, s, \mu, w$, and $\Gamma$ replaced by the former. Thus, by Lemma \ref{b0b1}, we obtain
	\begin{align*}
		(|D\tilde{u}| + \tilde{s})_{\tilde{B}_i} \le 4 \quad \text{for } i \ge 0.
	\end{align*}
	
	Then \eqref{normali} yields $(|Du| + s)_{\tilde{B}_i} \le 4M$ for $i \ge 0$. Thus, the lemma follows from the Lebesgue differentiation theorem.
\end{proof}
	
To estimate the term 
$
	\frac{c_8^p \Gamma r^\sigma M^\sigma}{1 - (1/2)^\sigma}
$
appearing in the proof of Lemma \ref{thmM1}, we require the following lemma.
	\begin{lemma}\label{rrb}
	Suppose that \eqref{0a}, \eqref{omegaa}, and \eqref{omegaz} hold. If $12r \in (0, \varepsilon R]$ and $B_{R}\subseteq \Omega$, then
		\begin{align*}
			r\fint_{B_{3r}} (|Du| + s) \, dx \le c_9 \left( \frac{r}{R} \right)^\beta R \left( \fint_{B_R} (|Du| + s) \, dx + \left[ \int_{0}^{R} \frac{|\mu|(B_\rho)}{\rho^{n-1}} \frac{d\rho}{\rho} + \int_{0}^{R} \frac{ {D}\Psi(B_\rho)}{\rho^{n-1}} \frac{d\rho}{\rho} \right]^{\frac{1}{p-1}} \right),
		\end{align*}
		where $c_9 = c_9(n, p, l, L) \ge 1$ and $\beta = \beta(n, p, l, L) \in (0, 1)$.
	\end{lemma}
	
	\begin{proof}
	Let $R_i = \varepsilon^i R$ and $B_i = B_{R_i}$. Then there exists $j \ge 1$ such that $12r \in (R_{j+1}, R_j]$, and Lemma \ref{uuosc} yields
		\begin{align}\label{rdux}
			r \fint_{B_{3r}} |Du| \, dx &\le c \fint_{B_{12r}} |u - (u)_{B_{12r}}| \, dx + c r \left[ \frac{|\mu|(B_{12r})}{r^{n-1}} \right]^{\frac{1}{p-1}} + c r \left[ \frac{ {D}\Psi(B_{12r})}{r^{n-1}} \right]^{\frac{1}{p-1}} \nonumber \\
			&\le c \fint_{B_j} |u - (u)_{B_j}| \, dx + c \left( \frac{r}{R} \right)^\beta R \left[ \int_{0}^{R} \frac{|\mu|(B_\rho)}{\rho^{n-1}} \frac{d\rho}{\rho} + \int_{0}^{R} \frac{ {D}\Psi(B_\rho)}{\rho^{n-1}} \frac{d\rho}{\rho} \right]^{\frac{1}{p-1}}.  
		\end{align}
		Next, we estimate the first term on the right-hand side. Since $\varepsilon \in (0, 1/16]$, it follows from Lemma \ref{uuosc} that there exists $\beta \in (0, 1)$ such that for every $i \ge 0$,
		\begin{align}\label{biib}
			\fint_{B_{i+1}} |u - (u)_{B_{i+1}}| \, dx \le (4^{\frac{1}{\beta}}\varepsilon)^\beta \fint_{B_i} |u - (u)_{B_i}| \, dx + c \varepsilon^{-n} R_i \left( \left[ \frac{|\mu|(B_i)}{R_i^{n-1}} \right]^{\frac{1}{p-1}} + \left[ \frac{ {D}\Psi(B_i)}{R_i^{n-1}} \right]^{\frac{1}{p-1}} \right) + c \delta^\beta R_i s.  
		\end{align}
	If $j=1$, the lemma follows directly from \eqref{rdux} and \eqref{biib}.
		For $j \ge 2$, multiplying \eqref{biib} by $(4^{\frac{1}{\beta}}\varepsilon)^{\beta(j-1-i)}$ for $i \in [1, j-1]$ and summing over $i$ from $1$ to $j-1$, we find that
		\begin{align*}
		&	\sum_{i=1}^{j-1} (4^{\frac{1}{\beta}}\varepsilon)^{\beta(j-1-i)} \fint_{B_{i+1}} |u - (u)_{B_{i+1}}| \, dx\\
			 &\le \sum_{i=1}^{j-1} (4^{\frac{1}{\beta}}\varepsilon)^{\beta(j-i)} \fint_{B_i} |u - (u)_{B_i}| \, dx \nonumber \\
			&\quad + c \sum_{i=1}^{j-1} \varepsilon^{\beta(j-1-i)-n} R_i \left( \left[ \frac{|\mu|(B_i)}{R_i^{n-1}} \right]^{\frac{1}{p-1}} + \left[ \frac{ {D}\Psi(B_i)}{R_i^{n-1}} \right]^{\frac{1}{p-1}} \right) + c \sum_{i=1}^{j-1} \varepsilon^{\beta(j-1-i)} \delta^\beta R_i s,
		\end{align*}
		which implies that
		\begin{align}\label{buj1} \nonumber
			\fint_{B_j} |u - (u)_{B_j}| \, dx &\le (4^{\frac{1}{\beta}}\varepsilon)^{\beta(j-1)} \fint_{B_1} |u - (u)_{B_1}| \, dx \\
			&\quad+ c \sum_{i=1}^{j-1} \varepsilon^{\beta(j-1-i)-n} R_i \left( \left[ \frac{|\mu|(B_i)}{R_i^{n-1}} \right]^{\frac{1}{p-1}} + \left[ \frac{ {D}\Psi(B_i)}{R_i^{n-1}} \right]^{\frac{1}{p-1}} \right) + c \sum_{i=1}^{j-1} \varepsilon^{\beta(j-1-i)} \delta^\beta R_i s.  
		\end{align}
	Since $R_i = \varepsilon^i R$ and $\beta \in (0, 1)$, a direct calculation yields
		\begin{align}\label{buj2}
			\sum_{i=0}^{j} \varepsilon^{\beta(j-1-i)-n} R_i = \sum_{i=0}^{j} \varepsilon^{\beta j - \beta - n + i(1-\beta)}  R \le \frac{\varepsilon^{\beta j - \beta - n}}{1 - \varepsilon^{1-\beta}} R \le c \left( \frac{R_j}{R} \right)^\beta R. 
		\end{align}
		Similarly,
		\begin{align}\label{buj3}
			c \sum_{i=1}^{j-1} \varepsilon^{\beta(j-1-i)} \delta^\beta R_i s \le c \left( \frac{R_j}{R} \right)^\beta R s. 
		\end{align}
	Since 
	\begin{align*}
		\frac{|\mu|(B_i)}{R_i^{n-1}} \le c \int_0^R \frac{|\mu|(B_\rho)}{\rho^{n-1}} \frac{d\rho}{\rho} \quad \text{and} \quad \frac{D\Psi(B_i)}{R_i^{n-1}} \le c \int_0^R \frac{D\Psi(B_\rho)}{\rho^{n-1}} \frac{d\rho}{\rho}
	\end{align*}
	hold for every $i \ge 1$, combining these with \eqref{buj1}, \eqref{buj2}, \eqref{buj3} and Poincaré's inequality, we obtain
		\begin{align*}
			\fint_{B_j} |u - (u)_{B_j}| \, dx &\le (4^{\frac{1}{\beta}}\varepsilon)^{\beta(j-1)} \fint_{B_1} |u - (u)_{B_1}| \, dx \\
			&\quad + c \left( \frac{R_j}{R} \right)^\beta R \left[ s + \left( \int_0^R \frac{|\mu|(B_\rho)}{\rho^{n-1}} \frac{d\rho}{\rho} + \int_0^R \frac{ {D}\Psi(B_\rho)}{\rho^{n-1}} \frac{d\rho}{\rho} \right)^{\frac{1}{p-1}} \right] \nonumber \\
			&\le c \left( \frac{R_j}{R} \right)^\beta R \left\{ \fint_{B_{1}} (|Du| + s) \, dx + \left( \int_0^R \frac{|\mu|(B_\rho)}{\rho^{n-1}} \frac{d\rho}{\rho} + \int_0^R \frac{ {D}\Psi(B_\rho)}{\rho^{n-1}} \frac{d\rho}{\rho} \right)^{\frac{1}{p-1}}  \right\} \\
			&\le c\left( \frac{r}{R} \right)^\beta R \left\{ \fint_{B_{R}} (|Du| + s) \, dx + \left( \int_0^R \frac{|\mu|(B_\rho)}{\rho^{n-1}} \frac{d\rho}{\rho} + \int_0^R \frac{ {D}\Psi(B_\rho)}{\rho^{n-1}} \frac{d\rho}{\rho} \right)^{\frac{1}{p-1}}  \right\}.	
		\end{align*}
	Combining this with \eqref{rdux}, we complete the proof of the lemma.
	\end{proof}

\begin{proof}[\textbf{Proof of Theorem \ref{thm}}]
 	By choosing $c_1 = c_8^p \left( \frac{\varepsilon^{3n}}{2^{30}} \right)^{-p}$ in \eqref{c11}, we set	
 		\begin{align}\label{mmt}
 		\tilde{M}^{p-1} = c_8^p \left( \frac{\varepsilon^{3n}}{2^{30}} \right)^{-p} \left[ \int_{0}^{R} \frac{|\mu|(B_\rho)}{\rho^{n-1}} \frac{d\rho}{\rho} + \int_{0}^{R} \frac{ {D}\Psi(B_\rho)}{\rho^{n-1}} \frac{d\rho}{\rho} + \left( \fint_{B_{R}} (|Du|+s) \, dx \right)^{p-1} \right]. 
 	\end{align}
 Next, let $r > 0$ be taken such that
 	\begin{align}\label{rrr}
 		r = \min \left\{ \left[ \frac{1 - (\frac{1}{2})^\sigma}{\Gamma (2 c_9 \tilde{M} R^{1-\beta})^\sigma} \cdot \left( \frac{\varepsilon^{3n}}{2^{30} c_8} \right)^{2p} \right]^{\frac{1}{\sigma\beta}}, \frac{\varepsilon R}{12} \right\},
 	\end{align}
 where the constants $c_8$ and $c_9$ are as specified in \eqref{constant} and Lemma \ref{rrb}, respectively. 
 	Then we obtain
 	\begin{align*}
 		\frac{1}{r^n} = \max \left[ \left( \frac{\Gamma (2c_9 R)^\sigma}{1 - (\frac{1}{2})^\sigma} \cdot \left( \frac{2^{30} c_8}{\varepsilon^{3n}} \right)^{2p} \right)^{\frac{n}{\sigma\beta}} \frac{\tilde{M}^{\frac{n}{\beta}}}{R^n}, \frac{12^n}{(\varepsilon R)^n} \right].
 	\end{align*}
 Since $p\geqslant2$ and $c_9 \ge 1$, it follows from Lemma \ref{rrb} and Lemma \ref{mmt} that
 	\begin{align*}
 		&\int_0^{3r} \frac{|\mu|(B_\rho)}{\rho^{n-1}} \frac{d\rho}{\rho} + \left( \fint_{B_{3r}} (|Du|+s) \, dx \right)^{p-1} + \int_0^{3r} \frac{ {D}\Psi(B_\rho)}{\rho^{n-1}} \frac{d\rho}{\rho} \nonumber \\
 	&	\le     \int_0^{R} \frac{|\mu|(B_\rho)}{\rho^{n-1}} \frac{d\rho}{\rho}  + \int_0^R \frac{ {D}\Psi(B_\rho)}{\rho^{n-1}} \frac{d\rho}{\rho}\\
 	&\quad + \left[ c_9 \left( \frac{R}{r} \right)^{1-\beta} \fint_{B_R} (|Du|+s) \, dx + \left[ \int_0^R \frac{|\mu|(B_\rho)}{\rho^{n-1}} \frac{d\rho}{\rho} + \int_0^R \frac{ {D}\Psi(B_\rho)}{\rho^{n-1}} \frac{d\rho}{\rho} \right]^{\frac{1}{p-1}} \right]^{p-1} \nonumber \\
 	&	\le \left[ 2 c_9 \left( \frac{R}{r} \right)^{1-\beta} \tilde{M} \right]^{p-1},
 	\end{align*}
 	which, together with \eqref{rrr}, yields
 	\begin{align*}
 	&	\frac{c_8^p \Gamma r^\sigma}{1 - (\frac{1}{2})^\sigma}   \left\lbrace c_8^p \left( \frac{\varepsilon^{3n}}{2^{30}} \right)^{-p} \left[ \int_0^{3r} \frac{|\mu|(B_\rho)}{\rho^{n-1}} \frac{d\rho}{\rho} + \int_0^{3r} \frac{ {D}\Psi(B_\rho)}{\rho^{n-1}} \frac{d\rho}{\rho} + \left( \fint_{B_{3r}} (|Du|+s) \, dx \right)^{p-1} \right]^{\sigma} \right\rbrace ^{\frac{1}{p-1}} \nonumber \\
 		&\le \frac{c_8^{2p} \Gamma}{1 - (\frac{1}{2})^\sigma}   \left( \frac{\varepsilon^{3n}}{2^{30}} \right)^{-p}   (2 c_9 \tilde{M} R^{1-\beta})^\sigma  r^{\sigma\beta} \nonumber \\
 		&\le \left( \frac{\varepsilon^{3n}}{2^{30}} \right)^p.
 	\end{align*}
 	Combining this with \eqref{c11}, we see that the hypotheses of Lemma \ref{thmM1} are satisfied; hence, we have
 	\begin{align*}
 		|Du(x_0)|  &\le c \left[ \int_0^{3r} \frac{|\mu|(B_\rho)}{\rho^{n-1}} \frac{d\rho}{\rho} + \int_0^{3r} \frac{ {D}\Psi(B_\rho)}{\rho^{n-1}} \frac{d\rho}{\rho} + \left( \fint_{B_{3r}} (|Du|+s) \, dx \right)^{p-1} \right]^{\frac{1}{p-1}} \nonumber \\
 		&\le c \fint_{B_R} (|Du|+s) \, dx + c \left( \int_0^R \frac{|\mu|(B_\rho)}{\rho^{n-1}} \frac{d\rho}{\rho} + \int_0^R \frac{ {D}\Psi(B_\rho)}{\rho^{n-1}} \frac{d\rho}{\rho} \right)^{\frac{1}{p-1}} \nonumber \\
 		& \quad + \left( \frac{c \Gamma R^\sigma}{1 - (\frac{1}{2})^\sigma} \right)^{\frac{n}{\beta \sigma}}\fint_{B_R} (|Du|+s) \, dx \\ \nonumber
 		 &\quad \cdot \left[ \left( \fint_{B_R} (|Du|+s) \, dx \right)^{ \frac{n}{\beta}} + \left( \int_0^R \frac{|\mu|(B_\rho)}{\rho^{n-1}} \frac{d\rho}{\rho} + \int_0^R \frac{ {D}\Psi(B_\rho)}{\rho^{n-1}} \frac{d\rho}{\rho} \right)^{\frac{n}{(p-1)\beta}} \right].
 	\end{align*} 
 This completes the proof of Theorem \ref{thm}.
 \end{proof}
 
Next, we proceed to the proofs of Theorem \ref{thm2} and Theorem \ref{thm3}. Under the assumptions of Theorem \ref{thm2} or Theorem \ref{thm3}, it follows from Theorem \ref{thm} that $Du \in L^\infty(B_{2R})$. We thus denote
\begin{align}
	\bar{M}^{p-1} = c_8^p \left( \frac{\varepsilon^{3n}}{2^{30}} \right)^{-p}   \left\{ \sup_{x \in B_R} \left[ \int_0^R \frac{|\mu|(B_\rho(x))}{\rho^{n-1}} \frac{d\rho}{\rho} + \int_0^R \frac{D\Phi(B_\rho(x))}{\rho^{n-1}} \frac{d\rho}{\rho} \right] +   \| (|Du| + s) \|_{L^\infty(B_{2R})}^{p-1} \right\}. \label{sum3}
\end{align}
\begin{proof}[\textbf{Proof of Theorem \ref{thm2}}]
 Fix a point $y \in B_R$ and a constant $m \in [3, \infty)$. According to \eqref{su2}, there exists a constant $\bar{r} = \bar{r}(\bar{M}, m, n, p, l, L, \mu, \omega, \Gamma, \sigma) \in (0, R]$ such that
\begin{align}
	\sup_{x \in B_R} \frac{1}{\bar{M}} \left[ \frac{|\mu|(B_r(x))}{r^{n-1}} + \frac{D\Phi(B_r(x))}{r^{n-1}} \right]^{\frac{1}{p-1}} + \Gamma (r\bar{M})^\sigma + [\omega(r)]^{\frac{2}{p}} \le \frac{1}{2^m}, \quad r \in (0, \bar{r}]. \label{sum4}
\end{align}
Next, we apply the normalization from Lemma \ref{thmM1}:
\begin{align}
	\left\{
	\begin{aligned}
		&\tilde{a}(x, z, \eta) = \frac{a(x, \bar{M}z, \bar{M}\eta)}{\bar{M}^{p-1}}, \quad \tilde{u}(x) = \frac{u(x)}{\bar{M}}, \quad \tilde{v}(x) = \frac{v(x)}{\bar{M}}, \quad \tilde{\psi}(x) = \frac{\psi(x)}{\bar{M}}, \quad \tilde{\mu} = \frac{\mu}{\bar{M}^{p-1}}, \\
		&\tilde{\Gamma} = \Gamma \bar{M}^\sigma, \quad \tilde{s} = s \bar{M}^{-1}, \quad D\tilde{\Psi}(B_r) = \int_{B_r} |\text{div } \tilde{a}(x, \tilde{u}, D\tilde{\psi})| \, dx.
	\end{aligned}
	\right. \label{sum5}
\end{align}
Since $y \in B_R$ and $B_r(y) \subseteq B_{2R}$, it follows from \eqref{sum3}, \eqref{sum4}, \eqref{sum5} and Theorem \ref{thm} that
\begin{align}
	(|D\tilde{u}| + \tilde{s})_{B_r(y)} \le 8 \quad \text{and} \quad \left[ \frac{|\tilde{\mu}|(B_r(y))}{r^{n-1}} + \frac{D\tilde{\Psi}(B_r(y))}{r^{n-1}} \right]^{\frac{1}{p-1}} + \tilde{\Gamma} r^\sigma + [\omega(r)]^{\frac{2}{p}} \le \frac{1}{2^m}. \label{sum6}
\end{align}
From Lemma \ref{du-6} and \eqref{sum6}, it follows that
\begin{align*}
	&\fint_{B_{\frac{\delta r}{2}}(y)} |D\tilde{u} - (D\tilde{u})_{B_{\frac{\delta r}{2}}(y)}| \, dx \\
	\le &\frac{1}{2^7} \fint_{B_r(y)} |D\tilde{u} - (D\tilde{u})_{B_r(y)}| \, dx + c \left[ \frac{|\tilde{\mu}|(B_r(y))}{r^{n-1}} + \frac{D\tilde{\Psi}(B_r(y))}{r^{n-1}} \right]^{\frac{1}{p-1}} + c \tilde{\Gamma} r^\sigma + c [\omega(r)]^{\frac{2}{p}} \\
	\le &\frac{1}{2^7} \fint_{B_r(y)} |D\tilde{u} - (D\tilde{u})_{B_r(y)}| \, dx + \frac{c}{2^m}.
\end{align*}
The arbitrariness of $r \in (0, \bar{r}]$, together with \eqref{sum5} and \eqref{sum6}, implies that
\begin{align*}
	\fint_{B_{(\frac{\delta}{2})^m r}(y)} |D\tilde{u} - (D\tilde{u})_{B_{(\frac{\delta}{2})^m r}(y)}| \, dx \le \frac{1}{2^{8m}} \fint_{B_r(y)} |D\tilde{u} - (D\tilde{u})_{B_r(y)}| \, dx + \frac{c}{2^m} \le \frac{c}{2^m}.
\end{align*}
Consequently, \eqref{sum5} yields
\begin{align*}
	\int_{B_{(\frac{\delta}{2})^m r}(y)} |Du - (Du)_{B_{(\frac{\delta}{2})^m r}(y)}| \, dx \le \frac{c\bar{M}}{2^m}, \quad \forall r \in (0, \bar{r}], \ y \in B_R.
\end{align*}
By the arbitrariness of $m \in [3, \infty)$, we conclude the proof of Theorem \ref{thm2}.
\end{proof}

To establish Theorem \ref{thm3}, we first prove the following lemma. 
\begin{lemma}\label{lemzz}
	Let $y \in B_R$, $r \in (0, R]$, and $m \in [3, \infty)$ be an integer. Suppose that 
	\begin{align}
		c_8^p \left[ \frac{1}{\bar{M}^{p-1}} \left( \int_0^{3r} \frac{|\mu|(B_\rho(y))}{\rho^{n-1}} \frac{d\rho}{\rho} + \int_0^{3r} \frac{D\Psi(B_\rho(y))}{\rho^{n-1}} \frac{d\rho}{\rho} \right) + \frac{\Gamma (r\bar{M})^\sigma}{1-(\frac{1}{2})^\sigma} + \int_0^{3r} \frac{\omega(\rho)}{\rho} d\rho \right] \le \left( \frac{\varepsilon^{5n}}{2^{2m+30}} \right)^p. \label{cmba}
	\end{align}
Then, for every $i, j \ge 2m$, we have
	\begin{align*}
		|(Du)_{B_{\varepsilon^i r}(y)} - (Du)_{B_{\varepsilon^j r}(y)}| \le \frac{5\bar{M}}{2^m},
	\end{align*}
where the constants $\varepsilon$ and $c_8$ are as specified in \eqref{constant}.
\end{lemma}

\begin{proof}
	Let $\hat{M} = \varepsilon^{2n} \bar{M}$, $r_i = \varepsilon^i r$ and $B_i(y) = B_{r_i}(y)$, it follows from \eqref{cmba} that
	\begin{align}
		c_8^p \left[ \frac{1}{\hat{M}^{p-1}} \left( \int_0^{3r} \frac{|\mu|(B_\rho(y))}{\rho^{n-1}} \frac{d\rho}{\rho} + \int_0^{3r} \frac{D\Psi(B_\rho(y))}{\rho^{n-1}} \frac{d\rho}{\rho} \right) + \frac{\Gamma (r\hat{M})^\sigma}{1-(\frac{1}{2})^\sigma} + \int_0^{3r} \frac{\omega(\rho)}{\rho} d\rho \right] \le \left( \frac{\varepsilon^{3n}}{2^{2m+30}} \right)^p \label{cm71},
	\end{align}
	which implies that
	\begin{align}
		& 2^{mp} c_7 \hat{M} \sum_{k=1}^\infty \left( \frac{|\mu|(B_k(y))}{\hat{M}^{p-1} r_k^{n-1}} + \frac{D\Psi(B_{k}(y))}{\hat{M}^{p-1} r_k^{n-1}} + \Gamma (r_k \hat{M})^\sigma + \omega(r_k) \right) \nonumber \\
		\le & \frac{2^{mp} \hat{M} c_7}{\varepsilon^{n-1}(-\log \varepsilon)} \left[ \frac{1}{\hat{M}^{p-1}} \left( \int_0^r \frac{|\mu|(B_\rho(y))}{\rho^{n-1}} \frac{d\rho}{\rho} + \int_0^r \frac{D\Psi(B_\rho(y))}{\rho^{n-1}} \frac{d\rho}{\rho} \right) + \frac{\Gamma (r \hat{M})^\sigma}{1-\left( \frac{1}{2}\right) ^\sigma} + \int_0^r \frac{\omega(\rho)}{\rho} d\rho \right] \le \frac{\varepsilon^n \bar{M}}{2^{m+7}}. \label{cm72}
	\end{align}
In view of \eqref{sum3}, we have
	\begin{align*}
		(|Du|+s)_{B_0(y)} + 2^9 \varepsilon^{-n} (|Du - (Du)_{B_0(y)}|)_{B_0(y)} &\le 2^{11} \varepsilon^{-n} (|Du|+s)_{B_0(y)} \\
		&\le \varepsilon^{2n} \bar{M}  \le \hat{M}.
	\end{align*}
Thus, one of the following cases must hold:

(1) There exists a constant $\ell \in [0, m-1]$ such that
\begin{align}
	\left\{
	\begin{aligned}
		&(|Du|+s)_{B_{\ell}(y)} + 2^9 \varepsilon^{-n} (|Du - (Du)_{B_{\ell}(y)}|)_{B_{\ell}(y)} \le \frac{\hat{M}}{2^\ell}, \\
		&(|Du|+s)_{B_{{\ell+1}}(y)} + 2^9 \varepsilon^{-n} (|Du - (Du)_{B_{{\ell+1}}(y)}|)_{B_{{\ell+1}}(y)} > \frac{\hat{M}}{2^{\ell+1}}.
	\end{aligned}
	\right. \label{lm-1}
\end{align}

(2) There exists a constant $\ell \ge m$ such that
\begin{align}
	\left\{
	\begin{aligned}
		&(|Du|+s)_{B_{i}(y)} + 2^9 \varepsilon^{-n} (|Du - (Du)_{B_{i}(y)}|)_{B_{i}(y)} \le \frac{\hat{M}}{2^m} \quad \text{for } i \in [m, \ell], \\
		&(|Du|+s)_{B_{{\ell+1}}(y)} + 2^9 \varepsilon^{-n} (|Du - (Du)_{B_{{\ell+1}}(y)}|)_{B_{{\ell+1}}(y)} > \frac{\hat{M}}{2^m}.
	\end{aligned}
	\right. \label{lmm}
\end{align}

(3) For every $\ell \ge m$, we have
\begin{align}
	(|Du|+s)_{B_{\ell}(y)} + 2^9 \varepsilon^{-n} (|Du - (Du)_{B_{\ell}(y)}|)_{B_{\ell}(y)} \le \frac{\hat{M}}{2^m}. \label{lm+1}
\end{align}

If \eqref{lm-1} holds, we set
\begin{align*}
	\widetilde{r}_k = \varepsilon^k r_\ell = r_{k+\ell} \quad \text{and} \quad \widetilde{B}_k(y) = B_{\widetilde{r}_k}(y) = B_{r_{k+\ell}}(y).
\end{align*}
Next, we apply the normalization from Lemma \ref{thmM1}:
\begin{align}
	\left\{
	\begin{aligned}
		&\widetilde{a}(x, z, \eta) = \frac{a(x, 2^{-\ell-1}\hat{M}z, 2^{-\ell-1}\hat{M}\eta)}{(2^{-\ell-1}\hat{M})^{p-1}}, \quad \widetilde{u}(x) = \frac{u(x)}{2^{-\ell-1}\hat{M}}, \quad \widetilde{v}(x) = \frac{v(x)}{2^{-\ell-1}\hat{M}}, \quad \widetilde{\psi}(x) = \frac{\psi(x)}{2^{-\ell-1}\hat{M}}, \\
		&\widetilde{\mu} = \frac{\mu}{(2^{-\ell-1}\hat{M})^{p-1}}, \quad \widetilde{\Gamma} = \Gamma (2^{-\ell-1}\hat{M})^\sigma, \quad \widetilde{s} = s (2^{-\ell-1}\hat{M})^{-1}, \quad D\widetilde{\Psi}( {B}_{\rho}(y)) \triangleq \int_{B_{\rho}(y)} |\text{div } \widetilde{a}(x, \tilde{u}, D\tilde{\psi})| \, dx.
	\end{aligned}
	\right. \label{lmt2}
\end{align}
It follows from \eqref{cm71} and \eqref{lm-1} that	
\begin{align*}
	\left\{
	\begin{aligned}
		& c_8^p \left( \int_0^{3\widetilde{r}_0} \frac{|\widetilde{\mu}|(B_\rho(y))}{\rho^{n-1}} \frac{d\rho}{\rho} + \int_0^{3\widetilde{r}_0} \frac{D\widetilde{\Psi}(B_\rho(y))}{\rho^{n-1}} \frac{d\rho}{\rho} + \frac{\widetilde{\Gamma} \widetilde{r}_0^\sigma}{1-(\frac{1}{2})^\sigma} + \int_0^{3\widetilde{r}_0} \frac{\omega(\rho)}{\rho} d\rho \right) \le \left( \frac{\varepsilon^{3n}}{2^{19}} \right)^p, \\
		& (|D\widetilde{u}|+\widetilde{s})_{\widetilde{B}_0(y)} + 2^9 \varepsilon^{-n} (|D\widetilde{u} - (D\widetilde{u})_{\widetilde{B}_0(y)}|)_{\widetilde{B}_0(y)} \le 2, \\
		& (|D\widetilde{u}|+\widetilde{s})_{\widetilde{B}_1(y)} + 2^9 \varepsilon^{-n} (|D\widetilde{u} - (D\widetilde{u})_{\widetilde{B}_1(y)}|)_{\widetilde{B}_1(y)} > 1.
	\end{aligned}
	\right.
\end{align*}
Combining the above inequalities with Lemma \ref{b0b1}, we find that if $k \ge 2$, then
\begin{align*}
	(|D\widetilde{u} - (D\widetilde{u})_{\widetilde{B}_{k+1}(y)}|)_{\widetilde{B}_{k+1}(y)} &\le \frac{1}{4}  (|D\widetilde{u} - (D\widetilde{u})_{\widetilde{B}_k(y)}|)_{\widetilde{B}_k(y)}\\
	&\quad + c_7 \left[ \frac{|\widetilde{\mu}|(\widetilde{B}_{k-2}(y))}{\widetilde{r}_{k-2}^{n-1}} + \frac{D\widetilde{\Psi}(\widetilde{B}_{k-2}(y))}{\widetilde{r}_{k-2}^{n-1}} + \widetilde{\Gamma} \widetilde{r}_{k-2}^\sigma + \omega(\widetilde{r}_{k-2}) \right].
\end{align*}
From \eqref{lmt2}, we see that if $k \ge \ell+2$, then
\begin{align}
	(|Du - (Du)_{B_{k+1}(y)}|)_{B_{k+1}(y)} \le & \frac{(|Du - (Du)_{B_k(y)}|)_{B_k(y)}}{4} \nonumber \\
	& + 2^{mp} c_7 \hat{M} \left[ \frac{|\mu|(B_{k-2}(y))}{\hat{M}^{p-1} r_{k-2}^{n-1}} + \frac{D\Psi(B_{k-2}(y))}{\hat{M}^{p-1} r_{k-2}^{n-1}} + \Gamma (\hat{M} r_{k-2})^\sigma + \omega(r_{k-2}) \right]. \label{lm+3}
\end{align}
Since $\ell \in [0, m-1]$ and $m \ge 3$, we have $m+1 \ge \ell+2$ and $2m \ge m+2$. Therefore, \eqref{lm+3} implies that
\begin{align*}
	& (|Du - (Du)_{B_{2m}(y)}|)_{B_{2m}(y)} \\
	\le & \frac{(|Du - (Du)_{B_{m+1}(y)}|)_{B_{m+1}(y)}}{4^{m-1}} + 2^{mp} c_7 \hat{M} \sum_{k=1}^\infty \left[ \frac{|\mu|(B_k(y))}{\hat{M}^{p-1} r_k^{n-1}} + \frac{D\Psi(B_k(y))}{\hat{M}^{p-1} r_k^{n-1}} + \Gamma (\hat{M} r_k)^\sigma + \omega(r_k) \right].
\end{align*}	
Thus, combining \eqref{sum3} and \eqref{cm72}, we obtain
	\begin{align}
		(|Du - (Du)_{B_{2m}}|)_{B_{2m}} \le \frac{\epsilon^n \bar{M}}{2^{m+7}} + \frac{\epsilon^n \bar{M}}{2^{m+7}} \le \frac{\epsilon^n \bar{M}}{2^{m+6}}. \label{lm+4}
	\end{align}
Since $\ell \in [0, m-1]$ and $m \ge 3$, it follows from \eqref{cm72}, \eqref{lm+3}, and \eqref{lm+4} that
	\begin{align*}
		\sum_{k=2m}^{\infty} &(|Du - (Du)_{B_k}|)_{B_k} \\
		&\le 2 (|Du - (Du)_{B_{2m}}|)_{B_{2m}} + 2^{mp+1} c_7 \hat{M} \sum_{k=1}^{\infty} \left[ \frac{|\mu|(B_k)}{\hat{M}^{p-1} r_k^{n-1}} + \frac{D\Psi(B_k)}{\hat{M}^{p-1} r_k^{n-1}} + \Gamma (\hat{M} r_k)^\sigma + \omega(r_k) \right] \\
		&\le \frac{\epsilon^n \bar{M}}{2^{m+4}}.
	\end{align*}	
	If $i, j \ge 2m$, then $$|(Du)_{B_i} - (Du)_{B_j}| \le \epsilon^{-n}    \sum_{k=2m}^{\infty} (|Du - (Du)_{B_k}|)_{B_k}.$$ Thus, from the above estimate and the fact that $B_i = B_{\epsilon^i}(y)$, the lemma follows whenever \eqref{lm-1} holds.
	
Next, assuming that \eqref{lmm} holds, it follows that
	\begin{align}
		\left\{
		\begin{aligned}
			&(|Du - (Du)_{B_l}|)_{B_l} \le \frac{\epsilon^n \hat{M}}{2^{m+9}} \le \frac{1}{16} \cdot \frac{\epsilon^{3n} \bar{M}}{2^{m+2}}, \\
			&(|Du - (Du)_{B_{l+1}}|)_{B_{l+1}} \le 2\epsilon^{-n} (|Du - (Du)_{B_l}|)_{B_l} \le \frac{2}{16} \cdot \frac{\epsilon^{2n} \bar{M}}{2^{m+2}}, \\
			&(|Du - (Du)_{B_{l+2}}|)_{B_{l+2}} \le 2\epsilon^{-n} (|Du - (Du)_{B_{l+1}}|)_{B_{l+1}} \le \frac{4}{16} \cdot \frac{\epsilon^{n} \bar{M}}{2^{m+2}}.
		\end{aligned}
		\right. \label{4.60}
	\end{align}
Then, by applying the normalization, we obtain
 \begin{align*}
 	\left\{
 	\begin{aligned}
 		&\widetilde{a}(x, z, \eta) = \frac{a(x, 2^{-m}\hat{M}z, 2^{-m}\hat{M}\eta)}{(2^{-m}\hat{M})^{p-1}}, \quad \widetilde{u}(x) = \frac{u(x)}{2^{-m}\hat{M}}, \quad \widetilde{v}(x) = \frac{v(x)}{2^{-m}\hat{M}}, \quad \widetilde{\psi}(x) = \frac{\psi(x)}{2^{-m}\hat{M}}, \\
 		&\widetilde{\mu} = \frac{\mu}{(2^{-m}\hat{M})^{p-1}}, \quad \widetilde{\Gamma} = \Gamma (2^{-m}\hat{M})^\sigma, \quad \widetilde{s} = s (2^{-m}\hat{M})^{-1}, \quad D\widetilde{\Psi}( {B}_{\rho}(y)) \triangleq \int_{B_{\rho}(y)} |\text{div } \widetilde{a}(x, \tilde{u}, D\tilde{\psi})| \, dx.
 	\end{aligned}
 	\right.  
 \end{align*}
 By a similar argument to the one leading to \eqref{lm+3}, we derive
  \begin{align*}
  	&(|Du - (Du)_{B_{k+1}}|)_{B_{k+1}} \\
  	&\quad \le \frac{1}{4} \cdot (|Du - (Du)_{B_k}|)_{B_k} + 2^{mp} c_7 \hat{M} \left[ \frac{|\mu|(B_{k-2})}{\hat{M}^{p-1} r_{k-2}^{n-1}} + \frac{D\Psi(B_{k-2})}{\hat{M}^{p-1} r_{k-2}^{n-1}} + \Gamma (\hat{M} r_{k-2})^\sigma + \omega(r_{k-2}) \right]
  \end{align*}
  for $k \ge \ell+2$. Thus, we conclude that
  \begin{align*}
  	\sum_{k=\ell+2}^{\infty} &(|Du - (Du)_{B_k}|)_{B_k} \\
  	&\le 2(|Du - (Du)_{B_{\ell+2}}|)_{B_{\ell+2}} + 2^{mp+1} c_7 \hat{M} \sum_{k=1}^{\infty} \left[ \frac{|\mu|(B_k)}{\hat{M}^{p-1} r_k^{n-1}} + \frac{D\Psi(B_k)}{\hat{M}^{p-1} r_k^{n-1}} + \Gamma (\hat{M} r_k)^\sigma + \omega(r_k) \right].
  \end{align*}
 In view of the above estimate, together with \eqref{cm72} and \eqref{4.60}, it follows that
  \begin{align}
  	\sum_{k=\ell}^{\infty} (|Du - (Du)_{B_k}|)_{B_k} \le \frac{\epsilon^n \bar{M}}{2^{m+2}}. \label{4.61}
  \end{align}
 Considering $i \in [m, \ell]$ and $i > \ell$ separately, we obtain
  \begin{align}
  	|(Du)_{B_i} - (Du)_{B_\ell}| &\le 2 \sup_{k \in [m, \ell]} |(Du)_{B_k}| + \sum_{k=\ell}^{\infty} |(Du)_{B_{k+1}} - (Du)_{B_k}| \nonumber \\
  	&\le 2 \sup_{k \in [m, \ell]} |(Du)_{B_k}| + \epsilon^{-n} \sum_{k=\ell}^{\infty} (|Du - (Du)_{B_k}|)_{B_k}. \label{4.62}
  \end{align}
 In view of $\epsilon \in (0, 1/16]$ and the estimates \eqref{lmm}, \eqref{4.61}, and \eqref{4.62}, the condition $i, j \ge 2m$ yields
  \begin{align*}
  	|(Du)_{B_i} - (Du)_{B_j}| &\le |(Du)_{B_i} - (Du)_{B_\ell}| + |(Du)_{B_j} - (Du)_{B_\ell}| \\
  	&\le 4 \sup_{k \in [m, \ell]} |(Du)_{B_k}| + 2\epsilon^{-n} \sum_{k=\ell}^{\infty} (|Du - (Du)_{B_k}|)_{B_k} \le \frac{5\bar{M}}{2^m}.
  \end{align*}
  Thus, given the relation $B_i = B_{\epsilon^i}(y)$, the lemma follows provided that \eqref{lmm} holds.
  
  Finally, if \eqref{lm+1} holds, then the lemma follows trivially. Therefore, by considering the three cases \eqref{lm-1}, \eqref{lm+1}, and \eqref{lmm}, we conclude that the lemma holds in all instances.
\end{proof}
\begin{proof}[\textbf{Proof of Theorem \ref{thm3}}]
  Fix an integer $m \ge 3$. In view of Theorem \ref{thm2} and the hypotheses of Theorem \ref{thm3}, there exists a radius $r = r(\bar{M}, m, n, p, l, L, \mu, \omega, \Gamma, \sigma) \in (0, R]$ such that
  \begin{align}
  	(|Du - (Du)_{B_r(y)}|)_{B_r(y)} \le \frac{\bar{M}}{2^m}, \quad \ \ \ \  for\  y \in B_R, \tau \in (0, r], \label{4.63}
  \end{align}
	\begin{align*}
		\sup_{x \in B_R} c_8^p \left[ \frac{1}{\bar{M}^{p-1}} \int_0^{3r} \frac{|\mu|(B_\rho(x)) \,  }{\rho^{n-1}  }\frac{d\rho}{\rho}+\int_0^{3r} \frac{D\Psi(B_\rho(x)) \,  }{\rho^{n-1}  }\frac{d\rho}{\rho} + \frac{\Gamma (r \bar{M})^\sigma}{1 - (1/2)^\sigma} + \int_0^{3r} \frac{\omega(\rho) \, d\rho}{\rho} \right] \le \left( \frac{\epsilon^{5n}}{2^{2m+30}} \right)^p.
	\end{align*}
It then follows from Lemma \ref{lemzz} that
	\begin{align}
		|(Du)_{B_{\varepsilon^i r}(y)} - (Du)_{B_{\varepsilon^j r}(y)}| \le \frac{5\bar{M}}{2^m} \quad (y \in B_R, i, j \ge 2m). \label{4.64}
	\end{align}
Since $m \ge 3$ was chosen arbitrarily, by combining \eqref{4.63} and \eqref{4.64}, we can conclude that every point in $B_R$ is a Lebesgue point of $Du$.
	
	 Let $\bar{r} = \epsilon^{2m} r$. Then, in view of \eqref{4.64}, it follows that 
	\begin{align}
		|(Du)_{B_{\bar{r}}(y)} - Du(y)| = \lim_{j \to \infty} |(Du)_{B_{\varepsilon^m r}(y)} - (Du)_{B_{\varepsilon^j r}(y)}| \le \frac{5\bar{M}}{2^m} \quad (y \in B_R). \label{4.65}
	\end{align}
Furthermore, there exists $\rho = \rho(n, \bar{r}) \in (0, \bar{r}]$ such that for any $y, z \in B_R$ with $|y - z| < \rho$, we have
	\begin{align}
		|(Du)_{B_{\bar{r}}(y)} - (Du)_{B_{\bar{r}}(z)}| \le \frac{\bar{M}(|B_{\bar{r}}(y) \setminus B_{\bar{r}}(z)| + |B_{\bar{r}}(z) \setminus B_{\bar{r}}(y)|)}{|B_{\bar{r}}|} \le \frac{\bar{M}}{2^m}. \label{4.66}
	\end{align}
By the arbitrariness of $y, z \in B_R$ satisfying $|y - z| < \rho$, we deduce from \eqref{4.65} and \eqref{4.66} that
	\begin{align*}
		|Du(y) - Du(z)| &\le |Du(y) - (Du)_{B_{\bar{r}}(y)}| + |(Du)_{B_{\bar{r}}(y)} - (Du)_{B_{\bar{r}}(z)}| + |(Du)_{B_{\bar{r}}(z)} - Du(z)| \\
		&\le \frac{11\bar{M}}{2^m}
	\end{align*}
for every $y, z \in B_R$ with $|y - z| < \rho$. Owing to the arbitrariness of $m \ge 3$, Theorem \ref{thm3} follows.
\end{proof}

\section*{Acknowledgments}The authors are supported by 
the   Sichuan Natural Science Foundation Youth Fund Project (Grant No. 2025ZNSFSC0799) and  the  National Natural Science Foundation of China (Grant No.~12571103, 12401122 and 12571226), Natural Science Foundation of Tianjin (Grant No. 25JCQNJC01400) and Young Scientific and Technological Talents (Level Three) in Tianjin.

\end{document}